\RequirePackage{fix-cm}
\documentclass[smallcondensed]{svjour3}       
\smartqed  
\usepackage{graphicx}
\usepackage{amscd,amssymb,epsfig,mathtools}
\newcommand\Conv{\operatorname{Conv}}

\newcommand{\diam}{\operatorname{diam}}

\newcommand{\ud}{\,d} 
 
\newcommand{\R}{\mathbb{R}}

\newcommand{\tir}[1]{\ensuremath{\overline {#1}}}

\def\whsq{\vbox to 5.8pt 
{\offinterlineskip\hrule 
\hbox to 5.8pt{\vrule height 
5.1pt\hss\vrule height 5.1pt}\hrule}} 
 
\def\Qed{{\hfill {\whsq}}} 
\def\<{\langle} 
\def\>{\rangle} 
\def\PP{{\mathop{{\rm I}\kern-.2em{\rm P}}\nolimits}} 
\def\FF{{\mathop{{\rm I}\kern-.2em{\rm F}}\nolimits}}   
\def\ZZ{{\mathop{{\rm I}\kern-.2em{\rm Z}}\nolimits}}

\begin{document}

\title{On the weak convergence of Monge-Amp\`ere measures for discrete convex mesh functions
}

\titlerunning{Monge-Amp\`ere measures for discrete convex mesh functions}        

\author{Gerard Awanou}


\institute{  Gerard Awanou \at
              Department of Mathematics, Statistics, and Computer Science (M/C 249), University of Illinois at Chicago,
Chicago, IL, 60607-7045\\              
              Tel.: +1-312-413-2167\\
              Fax: +1-312-996-1491 \\
              \email{awanou@uic.edu} \\
              ORCID https://orcid.org/0000-0001-9408-3078                  
              }

\date{Received: date / Accepted: date}

\maketitle

\begin{abstract}
To a mesh function we associate the natural analogue of the Monge-Amp\`ere measure. The latter is shown to be equivalent to the Monge-Amp\`ere measure of the convex envelope. We prove that the uniform convergence to a bounded convex function
of mesh functions 
implies the uniform convergence on compact subsets of their convex envelopes and hence the weak convergence of the associated Monge-Amp\`ere measures. We also give conditions 
for mesh functions to have a subsequence which converges uniformly to a convex function. Our result can be used to give alternate proofs of the convergence of some discretizations for the second boundary value problem for the Monge-Amp\`ere equation and was used for a recently proposed discretization of the latter. 
For mesh functions which are uniformly bounded and satisfy a convexity condition at the discrete level, we show that there is a subsequence which converges uniformly on compact subsets to a convex function. The convex envelopes of the mesh functions of the subsequence also converge uniformly on compact subsets. If in addition they agree with a continuous convex function on the boundary, the limit function is shown to satisfy the boundary condition strongly. 

\keywords{ discrete convex functions \and  weak convergence of measures \and Monge-Amp\`ere measure }
 \subclass{39A12 \and 35J60 \and 65N12 \and 65M06 }
\end{abstract}

\section{Introduction}

Let $\Omega$ be a nonempty bounded convex domain of $\R^d$ with boundary $\partial \Omega$ and let $u$ be a convex function on $\Omega$. 
For a locally integrable function $R$ on $\R^d$ such that $R>0$ on the image of $\Omega$ by the subdifferential of $u$, one associates, through the normal mapping, the $R$-curvature $\omega(R,u,.)$ of the convex function $u$ as a Borel measure. We define the analogous measures $\omega(R,u_h,.)$ for mesh functions $u_h$. We show that this notion of Monge-Amp\`ere measure coincides with the Monge-Amp\`ere measure of the convex envelope $\Gamma(u_h)$ of the mesh function $u_h$. We give in Theorem \ref{main0}, 
conditions under which the uniform convergence of mesh functions implies the weak convergence of the associated Monge-Amp\`ere measures. Theorem \ref{main0} is key to the proof of convergence of a recently proposed discretization for the second boundary value problem for the Monge-Amp\`ere equation \cite{Awanou-second-sym}. 
It can also be used, in the case $R\equiv 1$, to give an alternate proof for the convergence of the discretization proposed by Benamou and Duval in \cite{benamou2017minimal}. The method used in \cite{benamou2017minimal} does not seem to apply to the discretization we proposed in \cite{Awanou-second-sym}. We study the compactness of uniformly bounded discrete convex mesh functions 
c.f. Definition \ref{mass} 
below. When such mesh functions interpolate boundary values of a continuous convex function, we prove that a uniform limit function satisfies the boundary condition strongly. This result was used in \cite{DiscreteAlex2} to give a proof of convergence for a discretization, proposed by Benamou and Froese in \cite{BenamouFroese2017}, for the Dirichlet problem for a Monge-Amp\`ere equation with right hand side a sum of Dirac masses.


Our study is motivated by certain discretizations of the Monge-Amp\`ere equation \cite{DiscreteAlex2,Awanou-second-sym}: find a convex function $u$ on $\Omega$ such that
\begin{equation} \label{m1}
\omega(1,u,E) = \mu(E),
\end{equation}
for all Borels sets $E \subset \Omega$ and with appropriate boundary conditions. In \eqref{m1}, $R\equiv1$ and $\mu$ is a finite Borel measure. For illustration, let us assume that $\mu$ is absolutely continuous with density $f \geq 0$ and $f \in C(\tir{\Omega})$. In general a discretization of \eqref{m1} takes the form
\begin{equation} \label{m1d}
M^1_h[u_h] (x) = f(x),
\end{equation}
for a discrete Monge-Amp\`ere operator $M^1_h$, mesh points $x$ and an approximate mesh function $u_h$
which is discrete convex. 
It is then natural to consider the set function 
$$
M^1_h[u_h] (E) = \sum_{x \in E} M^1_h[u_h] (x),
$$
for a Borel set $E$. 
An approach for convergence is to determine a subsequence $u_{h_k}$ such that $M^1_{h_k}(u_{h_k}) (E)$ converges to $\omega(1,u,E)$ for all Borel sets $E$ with $\omega(1,u,\partial E)=0$. However, the set function $M^1_h[u_h]$ may not define a Borel measure and this makes such a weak convergence not straightforward.  When $\omega(1,u_h,E) \leq M^1_h[u_h] (E)$ for all Borel sets $E$, a step in proving the convergence of the discretization is to obtain a convergent subsequence for which $\omega(1,u_{h_k},E)$ converges to $\omega(1,u,E)$ for all Borel sets $E$ with $\omega(1,u,\partial E)=0$. 

We give conditions on the mesh functions $u_h$ such
that there is a subsequence $u_{h_k}$ for which $\Gamma(u_{h_k})$ converges uniformly on compact subsets to a convex function $v$. 
This implies, c.f. Theorem \ref{weak-cvg} below, that $\omega(1,\Gamma(u_{h_k}),.)=\omega(1,u_{h_k},.)$ weakly converges to $\omega(1,v,.)$. 
We prove that the subsequence $\Gamma(u_{h_k})$ has a further subsequence also denoted $\Gamma(u_{h_k})$ for which $u_{h_k}$ converges uniformly to $v$ on 
on compact subsets of $\Omega$.



The paper is organized as follows. In the next section we collect some notation used throughout the paper, recall the notion of $R$-Monge-Amp\`ere measure and introduce our discrete analogue. In section \ref{discrete}, we present the connection with the Monge-Amp\`ere measure of the convex envelope. This leads to weak convergence results for our discretization of the normal mapping.  In section \ref{interplay} we discuss the interplay between the uniform convergence of mesh functions and the uniform convergence of their convex envelopes. 
Compactness of mesh functions is discussed in section \ref{compact}. The treatment of Dirichlet boundary values is reduced to the convergence  of a finite difference scheme to the viscosity solution of the Laplace equation on a bounded Lipschitz domain. This requires a delicate treatment of barriers done in a recent work \cite{del2018convergence} which we review in an appendix for the convenience of the reader.

\section{Preliminaries}  \label{as}
For $x \in \R^d$ and $S \subset \R^d$ we denote by $d(x,S)$ the distance of $x$ to $S$, by $\diam(S)$ the diameter of $S$ and by $d(x,\partial \Omega)$ the distance of $x$ to $\partial \Omega$. For $x \in \R^d$, $||x||$ denotes the Euclidean norm of $x$. 
For a differentiable function $v$ at $x \in \R^d$, we denote by $D v(x)$ its gradient at $x$.

Let $h$ be a small positive parameter and let
$$
\mathbb{Z}^d_h = \{\, m h, m \in \mathbb{Z}^d \, \},
$$
denote the orthogonal lattice with mesh length $h$. 
Let also $(r_1, \ldots, r_d)$ denote the canonical basis of $\R^d$.
We define
\begin{equation*} 
\Omega_h  = \Omega \cap \mathbb{Z}_h^d.
\end{equation*}
For a function $u \in C(\Omega)$ 
its restriction on $\Omega_h$ is also denoted $u$ by an abuse of notation. For $x \in \Omega_h$ and $ e \in \mathbb{Z}^d$ let
$$
h^e_x = \sup \{ \, r h, r \in [0,1] \ \text{ and } \ x+rh e \in \tir{\Omega} \, \}.
$$
Next, let $V \subset \mathbb{Z}^d$ such that $\{ \, r_1,\ldots,r_d\, \} \subset V$ and such that for $e \in V$, $-e \in V$. Furthermore, we assume that for each $e \in V$, one can find a basis $\{ \, e_1,\ldots,e_d\, \}$ of $ \mathbb{Z}^d$ such that $e_1=e$. We define 
\begin{equation} \label{boundary-domain}
\partial \Omega_h = \{ \, x \in \partial \Omega,  \exists y \in \Omega_h  \ \text{and} \ e \in V \text{ such that } x= y + h^e_y e   \,  \}.
\end{equation}
For the convergence study in this paper, we want $V \to \mathbb{Z}^d$. Thus, later we will simply choose $V=\mathbb{Z}^d$. 
As in \cite{Nochetto19,NeilanZhang19}  we denote by $\mathcal{N}_h$ the set of nodes, i.e. $\mathcal{N}_h=\Omega_h \cup \partial \Omega_h$. Let $\Conv (S)$ denote the convex hull of the set $S$. Since $\Omega$ is convex and $V$ contains the elements of the canonical basis of $\R^d$, $\mathcal{N}_h \cap \partial \Conv (\mathcal{N}_h) = \partial \Omega_h$.  

We let  $\mathcal{U}_h$ denote the linear space of mesh functions, i.e. real-valued functions defined on $\mathcal{N}_h$.

\subsection{R-curvature of convex functions} \label{as}
The material in 
this subsection is mostly taken from  \cite{Bakelman1994,Guti'errez2001} to which we refer for proofs. Let $\Omega$ be an open subset of $\R^d$ and let us denote by $\mathcal{P}(\R^d)$ the set of subsets of $\R^d$. 
\begin{definition}
Let $u: \Omega \to \R$.  The normal mapping of $u$, or subdifferential of $u$ is the set-valued mapping $\partial u: \Omega \to \mathcal{P}(\R^d)$ defined by
\begin{align} \label{normal-mapping}
\partial u (x_0) = \{ \, p \in \R^d: u(x) \geq u(x_0) + p \cdot (x-x_0), \, \text{ for all } \, x \in \Omega\,\}.
\end{align}
\end{definition}
Let $|E|$ denote the Lebesgue measure of the measurable subset $E \subset \Omega$. 
For $E \subset \Omega$, we define 
$$
\partial u(E) = \cup_{x \in E} \partial u(x).
$$ 

\begin{theorem}[\cite{Guti'errez2001} Theorem 1.1.13]
If $u$ is continuous on $\Omega$, the class
 \begin{align*}
\mathcal{S} = \{\, E \subset \Omega, \partial u(E) \, \text{is Lebesgue measurable}\, \},
\end{align*}
is a Borel $\sigma$-algebra. 
\end{theorem}

Let $R$ be a locally integrable function on $\R^d$ such that $R>0$ on $\partial u(\Omega)$.  Without loss of generality, we will assume that $R=0$ on $\R^d \setminus \partial u(\Omega)$.
The $R$-curvature of the convex function $u$ is defined as the set function
$$
\omega(R,u,E) = \int_{\partial u(E)} R(p) \ud p,
$$
and can be shown to be a Radon measure on $\mathcal{S}$ \cite[Theorem 1.1.13]{Guti'errez2001}. The set function $\omega(R,u,.)$ is also referred to as the $R$-Monge-Amp\`ere measure associated with the convex function $u$.

\begin{definition}
A sequence $\mu_n$ of Borel measures converges to a Borel measure $\mu$ if and only if $\mu_n(B) \to \mu(B)$ for any Borel set $B \subset \Omega$ with $\mu(\partial B)=0$.
\end{definition}

We note that there are several equivalent definitions of weak convergence of measures which can be found for example in \cite[Theorem 1, section 1.9]{Evans-Gariepy}. It is known that the uniform limit of convex functions is convex. We have \cite[Theorem 9.1]{Bakelman1994}

\begin{theorem} \label{weak-cvg}
Let $u_n$ be a sequence of convex functions on $\Omega$ such that $u_n$ converges to $u$ uniformly on compact subsets of $\Omega$, then $\omega(R,u_n,.)$ weakly converges to $\omega(R,u,.)$.

\end{theorem}

We will make use of the following observation

\begin{lemma} \label{inter-meas}
Let $\mu$ be a Borel measure and $H_i, i=1\ldots,N$ be a sequence of Borel sets with pairwise intersection of $\mu$-measure zero. Then $\mu(\cup_{i=1}^N H_i)=\sum_{i=1}^N \mu(H_i)$.

\end{lemma}

\begin{proof} We have
$\cup_{i=1}^N H_i = H_1 \cup (H_2 \setminus H_1) \cup (H_3 \setminus (H_2 \cup H_1) ) \cup \ldots,$
with the sets on the right hand side disjoints. Moreover
$$
H_j = [H_j \cap (H_{j-1} \cup H_{j-2} \cup \ldots \cup H_1)] \cup [H_j \setminus (H_{j-1} \cup H_{j-2} \cup \ldots \cup H_1)].
$$
But $\mu(H_j \cap (H_{j-1} \cup H_{j-2} \cup \ldots \cup H_1) \leq \mu \big(\cup_{k=1}^{j-1} H_j \cap H_k \big) \leq \sum_{k=1}^{j-1} \mu(H_j \cap H_k) =0$ and hence
$$
\mu(H_j)  = \mu(H_j \setminus (H_{j-1} \cup H_{j-2} \cup \ldots \cup H_1).
$$
This implies that $\mu(\cup_{i=1}^N H_i ) = \sum_{i=1}^N \mu(H_i)$. 
\Qed
\end{proof}

It will be convenient to consider an extension to $\R^d$ of a convex function on $\Omega$ with the same subdifferential when $\partial u(\Omega)$ is convex. The construction is standard and can be found for example on \cite[section 2]{Caffarelli92a}. 
If $u$ is a convex function on $\Omega$, we extend $u$ to $\R^d$ using
\begin{equation} \label{extension}
\tilde{u}(x) = \sup \{ \,   u(y) + q \cdot (x-y), y \in \Omega, q \in \partial u(y) \, \}.
\end{equation}
We denote by $\chi_{u}$ the subdifferential of the extended function on $\R^d$. 

A convex function $v$ on $\R^d$ is said to be proper if $v(x_0) < \infty$ for some $x_0 \in \R^d$ and $v(x)>-\infty$ for all $x \in \R^d$. We will often use the following lemma \cite[Corollary 6.10]{Arutyunov} and \cite[Theorem 9.9]{Arutyunov}

\begin{lemma} \label{continuity-convex}
A proper convex function $v$ which is bounded above on $\Omega$ is continuous on  $\Omega$. Moreover  for all $x\in \Omega$, $\partial v(x)\neq \emptyset$.  
\end{lemma}

\begin{lemma} \label{extension-convex}
The extension $\tilde{u}$ of the convex function $u$ on $\Omega$ defined by \eqref{extension} is convex on $\R^d$ and for all $x \in \Omega$, $\partial u (x) = \chi_u ( x)$. If $u$ is a proper convex function bounded on $\Omega$ with $\partial u(\Omega)$ bounded, then $\tilde{u}$ is continuous on $\R^d$. For $u \in C(\tir{\Omega})$ convex, we have $\tilde{u}=u$ on $\tir{\Omega}$.
\end{lemma}

\begin{proof}
Note that $\tilde{u}$ is convex as a supremum of convex functions and $\tilde{u}=u$ on $\Omega$ since for $x, y \in \Omega$ and $q \in \partial u(y)$, $u(x) \geq u(y)+q \cdot (y-x)$ with equality at $x=y$.

Next,  for $x \in \Omega$,  since $\tilde{u}=u$ on $\Omega$, we immediately get $ \chi_u ( x) \subset \partial u (x)$. Let $p \in \partial u (x_0), x_0 \in \Omega$. 
we have by definition of $\tilde{u}$, $\tilde{u}(x) \geq u(x_0) + p \cdot (x-x_0)=\tilde{u}(x_0) + p \cdot (x-x_0)$ for all $x \in \R^d$.
Thus $ \partial u (x_0) \subset \chi_u(x_0)$. This proves that for all $x \in \Omega$, $\partial u (x) = \chi_u ( x)$.


Let $x \in \R^d$ and $U$ a bounded open set such that $x \in U$. Put $L(x) = u(y)+ q \cdot (x-y)$ for $y \in \Omega$ and $q \in \partial u(y)$. If $u$ is bounded on $\Omega$ with $\partial u(\Omega)$ bounded, $L$ is bounded from below on $U$ and so is $\tilde{u}$. Since $\tilde{u}$ is an extension of $u$, $\tilde{u}(x_0) <\infty$ for some $x_0 \in \Omega$. This implies that $\tilde{u}$ is a proper convex function. We conclude by Lemma \ref{continuity-convex} that for $u$ proper, convex and bounded on $\Omega$, $\tilde{u}$ is continuous on $\R^d$.

If $u \in C(\tir{\Omega})$ is convex, it is proper, convex and bounded on $\Omega$. Thus $\tilde{u}=u$ on $\tir{\Omega}$ since 
$\tilde{u}$ is continuous on $\R^d$. 

\Qed
\end{proof}

\begin{lemma} \label{ext-subd}
Let $u$ be a proper bounded convex  function on $\Omega$.
Assume that $\Omega^* = \partial u (\Omega)$ is bounded. Then $\tir{\Omega^*} \subset  \chi_u ( \tir{\Omega}) \subset \chi_u(\R^d) \subset \Conv ( \tir{\Omega^*})$.  
\end{lemma}

\begin{proof}

We have by Lemma \ref{extension-convex}, $\chi_u(\Omega) = \partial u (\Omega)  = \Omega^* \subset 
\tir{\Omega^*}$.  Next, for $x \in \R^d \setminus \Omega$, we have 
$$
\tilde{u}(x) = \lim_{n \to \infty} u(y_n) + q_n \cdot (x-y_n)= \lim_{n \to \infty} \tilde{u}(y_n) + q_n \cdot (x-y_n),
$$
for $y_n \in \Omega$ and $q_n \in \partial u(y_n) = \chi_u(y_n)$. Since $\partial u (\Omega)$ is bounded, both $y_n$ and $q_n$ are bounded sequences. Up to a subsequence $y_n \to y$ and $q_n \to q$ for some $y \in \tir{\Omega}$ and 
$q \in \tir{\Omega^*}$. It is immediate that $q \in \chi_u(y)$. By Lemma \ref{extension-convex}, $\tilde{u}$ is continuous on $\R^d$. Therefore for $x \in \R^d$
$$
\tilde{u}(x) = \max \{ \,  \tilde{u}(y) + q \cdot (x-y), y \in \tir{\Omega}, q \in \chi_u(y) \cap \tir{\Omega^*}  \, \}.
$$
We claim that at any point $x$ where $\tilde{u}$ is differentiable, $D \tilde{u}(x) \in  \tir{\Omega^*}$. Let  $y \in \tir{\Omega}$ and $q \in \chi_u(y) \cap \tir{\Omega^*}$ such that  $\tilde{u}(x) = \tilde{u}(y) + q \cdot (x-y)$. It is enough to show that $q \in \chi_{u}(x)$, which implies by the differentiability of $\tilde{u}$ at $x$, $\chi_{u}(x) = \{ \,q \, \} = \{ \, D \tilde{u}(x) \, \}$. For all $z \in \R^d$, we have
\begin{align*}
\tilde{u}(z) \geq \tilde{u}(y) + q\cdot(z-y)= \tilde{u}(x) - q \cdot (x-y) + q\cdot(z-y)=  \tilde{u}(x) + q \cdot (z-x) . 
\end{align*}
The claim is proved.

Note that the domain of $\tilde{u}$ is $\R^d$ and the normal cone to $\R^d$ at $x$, i.e. $\{ \, p \in \R^d, p \cdot (x-y) \leq 0, \forall y \in \R^d\, \}= \{ \, 0 \, \}$. Also, $\tilde{u}$ is lower semicontinuous as the supremum of a family of continuous functions. It is therefore a closed function, i.e. its epigraph $\{ \, (x,w) \in \R^d, \tilde{u}(x) \leq w \, \} $ is closed. By \cite[Theorem 25.6]{Rockafellar70}, for $x \in \R^d$, $\chi_u(x)$ is the closure of the convex hull of the set $S(x)$ of limits of sequences of the form $D \tilde{u}(x_k), x_k \to x$ and $\tilde{u}$ differentiable at $x_k$. The theorem is stated in terms of the sum of $S(x)$ and the normal cone to the domain of $\tilde{u}$ at $x$.  

We have shown that $S(x) \subset \tir{\Omega^*}$. Thus $\chi_u(x) = \tir{\Conv S(x)} \subset \tir{\Conv ( \tir{\Omega^*})} = \Conv ( \tir{\Omega^*})$ where we use the fact that the convex hull of a compact subset of $\R^d$ is closed. We conclude that $\chi_u(\R^d) \subset \Conv ( \tir{\Omega^*})$. 

Next, we show that $\tir{\Omega^*} \subset \chi_u ( \tir{\Omega})$. Let $p \in \tir{\Omega^*}$ and $p_n \in \Omega^*$ such that $p_n \to p$. Furthermore, let $x_n \in \Omega$ such that $p_n \in \partial u(x_n)=\chi_u(x_n)$. We have for all $z \in \R^d$
$\tilde{u}(z) \geq \tilde{u}(x_n) + p_n \cdot (z-x_n)$. Up to a subsequence $x_n \to x, x \in \tir{\Omega}$ and we have $p \in \chi_u(x)$, using the continuity of $\tilde{u}$. We conclude that
$$
\tir{\Omega^*} \subset  \chi_u ( \tir{\Omega}) \subset \chi_u(\R^d) \subset \Conv ( \tir{\Omega^*}).
$$
\Qed

\end{proof}


\subsection{Convex envelopes}

For a function $u: \tir{\Omega} \to \R$, recall that the convex envelope $\Gamma(u)$ of $u$ is the largest convex function majorized on $\tir{\Omega}$ by $u$. 
 If we assume that $u(x) \geq C$ for all $x\in \tir{\Omega}$ for a constant $C$, then $\Gamma(u) \geq C$ and thus $\Gamma(u)$ is a proper convex function on $\tir{\Omega}$. It can be shown that for all $x \in \Omega$
$$
\Gamma(u)(x) = \sup_{L \text{ affine} }\{ \, L(x): L(y) \leq u(y) \ \forall y \in \tir{\Omega}
\, \},
$$
using supporting hyperplanes $L(y) = v(x) + p \cdot (y-x)$ for $p \in \partial \Gamma(u)(x)$, $x \in \Omega$.

As in \cite{Nochetto19,NeilanZhang19} we consider the convex envelope of the mesh function $u_h$ defined 
by
\begin{align} \label{def-cvx}
\Gamma(u_h)(x) = \sup_{L \text{ affine} }\{ \, L(x): L(y) \leq u_h(y)  \text{ for all } y \in \mathcal{N}_h
\, \}.
\end{align}
Without loss of generality we may assume that $\Omega_h$ is non empty. Thus for $x_0 \in \Omega_h$ we have
$\Gamma(u_h)(x_0) \leq u_h(x_0)$. Since $\Omega_h$ is finite, we can find $\alpha \in \R$ such that $\alpha \leq u_h(x) \ \forall x \in \Omega_h$. Thus $\Gamma(u_h)(x) \geq \alpha \ \forall x \in \R^d$. 
As the supremum of convex functions, $\Gamma(u_h)$ is convex on $\R^d$. We conclude that $\Gamma(u_h)$ is a proper convex function.
We have 
\begin{equation} \label{g-u}
\Gamma(u_h)(x) \leq u_h(x) \ \forall x \in \mathcal{N}_h. 
\end{equation}
If $x \in \Conv (\mathcal{N}_h)$, we can write $x=\sum_{i=1}^N \lambda_i y_i$ with $y_i \in  \mathcal{N}_h$ and $\lambda_i \in [0,1]$ for all $i$ and an integer $N$. We then have $\Gamma(u_h)(x) \leq \sum_{i=1}^N \lambda_i \Gamma(u_h)(y_i) 
\leq \sum_{i=1}^N \lambda_i u_h(y_i)$. This shows that $\Gamma(u_h)$ is bounded on $\Conv (\mathcal{N}_h)$. Note that the definition \eqref{def-cvx} of the convex envelope allows an ''infinite slope'' at points of $\R^d$ not in $\Conv (\mathcal{N}_h)$. 

We will denote by $\widetilde{\Gamma}(u_h)$ the extension to $\R^d$ of $\Gamma_h(u)$ with the procedure described by \eqref{extension}. See Figure \ref{extension-fig}.

\begin{figure}[tbp]
\begin{center}
\includegraphics[angle=0, height=4.5cm]{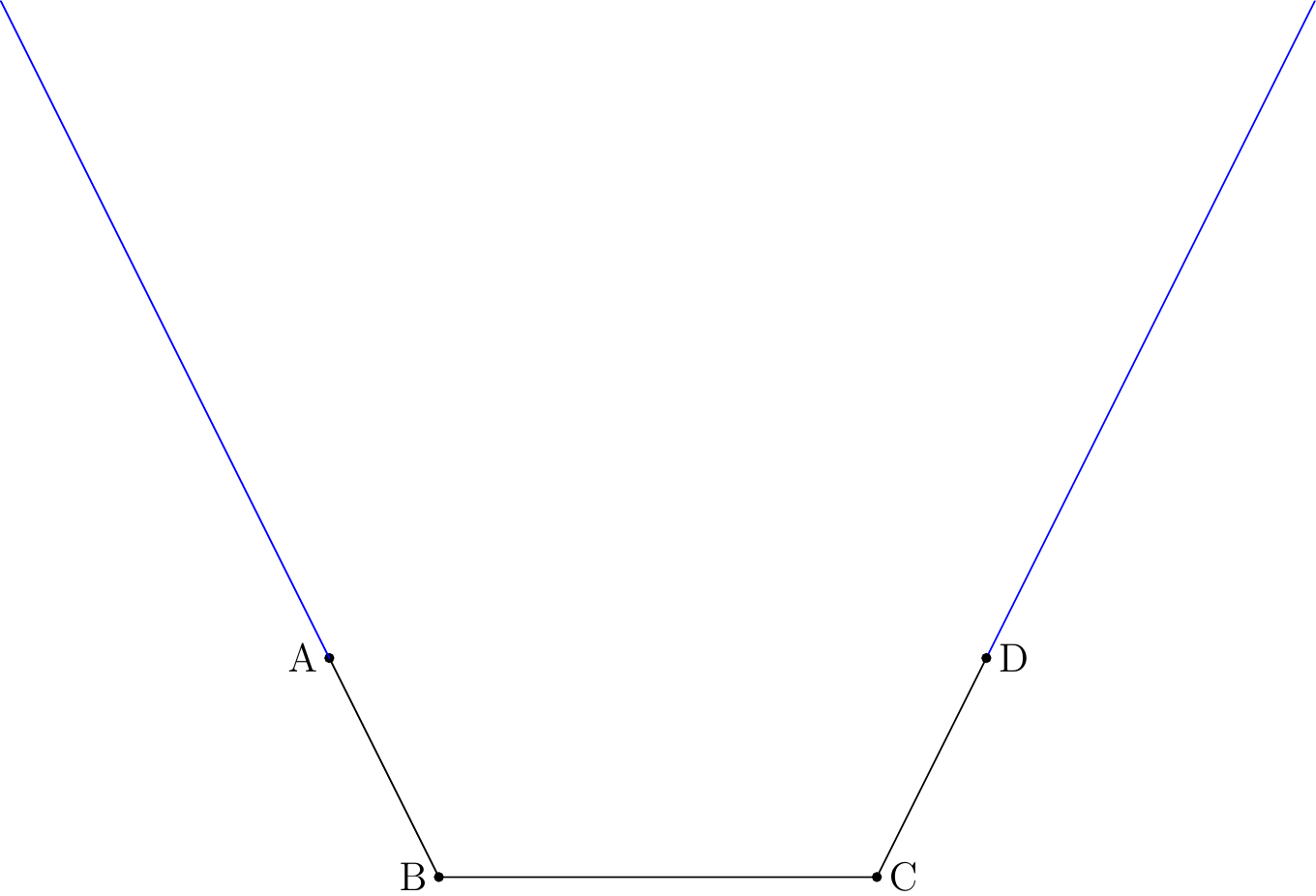}
\end{center}
\caption{ Convex envelope of the points $A(-1.5,1)$, $B(-1,0)$, $C(1,0)$ and $D(1.5,1)$ as a piecewise linear convex function on $[-1.5,1.5]$. The canonical extension as a piecewise linear convex function on $\R$ is shown. Without the extension, the convex envelope is infinite on $\R \setminus [-1.5,1.5]$. } 
\label{extension-fig}
\end{figure}

As a proper bounded convex function on $\Conv (\mathcal{N}_h)$, $\Gamma(u_h)$ is continuous on the interior of $\Conv (\mathcal{N}_h)$ by Lemma \ref{continuity-convex}. Now, recall that $\Conv (\mathcal{N}_h)$ is polyhedral. By \cite[Theorem 10.2]{Rockafellar70}, $\Gamma(u_h)$ is upper semicontinuous at any point of $\partial \Conv (\mathcal{N}_h)$. As a supremum of affine functions, $\Gamma(u_h)$ is lower semicontinuous on $\Conv (\mathcal{N}_h)$. We conclude that $\Gamma(u_h)$ is continuous and by Lemma \ref{extension-convex} $\Gamma(u_h)=\widetilde{\Gamma}(u_h)$ on $\Conv (\mathcal{N}_h)$.


\begin{lemma} \label{cv-bd0}
For a mesh function $u_h$ such that $u_h=g$ on $\partial \Omega_h$ for a convex function $g$ on $\R^d$, we have  $\Gamma(u_h) = u_h$ on $\partial \Omega_h$. 
\end{lemma}

\begin{proof}

Recall that a face $F$ of a convex set $C$ is such that if $x \in F$ and $y, z \in C$ such that $x=\lambda y + (1-\lambda ) z$, $0<\lambda<1$, we have $y, z \in F$.

Let $x \in \partial \Omega_h$ and $F$ a face of $Conv (\mathcal{N}_h)$ such that $x \in F$. By  \cite[Theorem 1]{Tawarmalani}, 
$F = \Conv(B)$ for $B \subset \partial \Omega_h$. Let $\Gamma(u_h|_B)$ denote the convex envelope of $u_h$ over $B$, i.e. the largest convex function majorized by $u_h=g$ on $B$. Since $g$ is convex on $\R^d$, we have $\Gamma(u_h|_B)=g$ on $B$.  By \cite[Corollary 2]{Tawarmalani}, the restriction of $\Gamma(u_h)$ to $F$ is equal to $\Gamma(u_h|_B)$. This proves that $\Gamma(u_h) = u_h=g$ on $\partial \Omega_h$. \Qed

\end{proof}

We will also need the discrete convex envelope $\Gamma_h(u)$ of a function $u$ on $\tir{\Omega}$, i.e. for $x \in \tir{\Omega}$
\begin{equation} \label{disc-cv}
\Gamma_h(u)(x) = \sup_{L \text{ affine} }\{ \, L(x): L(y) \leq u(y) \ \forall y \in \mathcal{N}_h \, \}.
\end{equation} 
In other words, $\Gamma_h(u)(x)$ is the convex envelope of the restriction to $ \mathcal{N}_h$ of the function $u$. We reserve the notation $\Gamma(u)$ for the convex envelope of a function $u$ on $\tir{\Omega}$.

\begin{theorem} \label{discrete-cvx-envelope-thm}
Let $u$ be a  function on $\tir{\Omega}$ which is convex on $\tir{\Omega}$. Then
$$
u(y) = 
\Gamma(u)(y) = \Gamma_h(u)(y), y \in  \mathcal{N}_h. 
$$
\end{theorem}

\begin{proof}
By definition  $\Gamma_h(u)(y) \leq u(y), y \in  \mathcal{N}_h$. Note that, a priori, this inequality does not necessarily hold on $\tir{\Omega}$.
If $L$ is affine with $L(y) \leq u(y)$ for all $y \in \tir{\Omega}$, we also have $L(y) \leq u(y)$ for all $ y \in \mathcal{N}_h$. Thus $ \Gamma(u)(x) \leq \Gamma_h(u)(x)$ for all $x \in \mathcal{N}_h$. But since $u$ is convex on $\tir{\Omega}$, $\Gamma(u)(x)=u(x)$ for all $x \in \tir{\Omega}$. Thus for all $y \in \mathcal{N}_h$, $u(y) = \Gamma(u)(y) \leq \Gamma_h(u)(y) \leq u(y)$. 
This proves the result. \Qed

\end{proof}

\subsection{Discrete normal mapping and convexity} \label{bo-mesh}

For 
a mesh function $u_h \in \mathcal{U}_h$, the discrete normal mapping of $u_h$ at the point $x \in \Omega_h$ 
is defined as
\begin{align*}
\partial_h u_h(x) = \{ \, p \in \R^d, u_h(y) \geq u_h(x) + p \cdot (y-x) \, \forall  y \in \mathcal{N}_h
\, \}.
\end{align*}
We note that $\partial_h u_h(x)$ may be empty. As in \cite{Mirebeau15}, for $e \in \mathbb{Z}^d, v_h \in \mathcal{U}_h$ and $x \in \Omega_h$ we  define
\begin{equation} \label{Delta-e}
\Delta_e u_h (x) = \frac{2}{h^e_x + h^{-e}_x} \bigg( \frac{u_h(x+ h^e_x e ) - u_h(x)}{h^e_x}  + \frac{u_h(x- h^{-e}_x e ) - u_h(x)}{h^{-e}_x} \bigg).
\end{equation}
Recall that the set of directions $V$ used in the next definition was used for the definition of $\partial \Omega_h$.
\begin{definition} \label{discrete-convex}
We say that a mesh function $u_h$ is {\it discrete convex} if and only if $\Delta_e u_h(x) \geq 0$ for all $x \in \Omega_h$ and  $e \in V \subset \mathbb{Z}^d$. 
\end{definition}

The above definition is motivated by discretizations of the Monge-Amp\`ere operator for smooth convex functions, i.e. 
$\omega(1,u,E) = \int_E \det D^2 u(x) \ dx$. Here $\det D^2 u(x)$ is the determinant of $D^2 u(x)=\bigg( \partial^2 u(x)/\partial x_i \partial x_j \bigg)_{i,j=1,\ldots, d}$, $x=(x_1,\ldots,x_d)$. For the discretization proposed in \cite{Oberman2010a} the discrete Monge-Amp\`ere operator is taken as
$$
\mathcal{M}_h[u_h](x) = \inf_{ (e_1,\ldots,e_d) \in W } \prod_{i=1}^d \frac{ \max\{ \, \Delta_{e_i} u_h(x),0 \, \}}{|| e_i ||^2}, x \in \Omega_h,
$$
where 
$$
W = \{ \, (e_1,\ldots,e_d), e_i \in \mathbb{Z}^d, \, i=1,\ldots,d,  (e_1,\ldots,e_d) \, \text{is an orthogonal basis of} \, \R^d \, \}.
$$
If $x \in \Omega_h, f(x)>0$ and $\mathcal{M}_h[u_h](x)=f(x)$,  it is necessary to have $\Delta_{e_i} u_h(x)>0$ for all $i$ with
$(e_1,\ldots,e_d) \in W$. 


One may also define discrete convexity by requiring that $\partial_h u_h(x)\neq \emptyset$ for all $x \in \Omega_h$. For $p \in \partial_h u_h(x)$, we have $( u_h(x) - u_h(x- h^{-e}_x e ) )/h^{-e}_x \leq p \cdot e \leq (u_h(x+ h^e_x e ) - u_h(x))/h^{e}_x$. This implies that $\Delta_e u_h (x) \geq 0$. However, a mesh function may satisfy $\partial_h u_h(x) = \emptyset$ with $\Delta_e u_h (x) \geq 0$ for all $e \in \mathbb{Z}^d$. 

\section{Monge-Amp\`ere measures for mesh functions}  \label{discrete}

In this section we relate the discrete normal mapping of a mesh function to the  normal mapping of its convex envelope. 
Recall that for $x \in \R^d$, $\chi_{u}(x)$ denotes the subdifferential of the  extension of the convex function $u$. Also, $ \widetilde{\Gamma}(u_h)$ denotes the convex extension of $\Gamma(u_h)$.

\begin{lemma} \label{equi0}
If $x \in \Conv (\mathcal{N}_h)$,  for $p \in  \chi_{\Gamma(u_h)} (x)$ there exists $y \in \mathcal{N}_h$ such that $p \in \chi_{\Gamma(u_h)} (x) \cap \chi_{\Gamma(u_h)} (y)$ and $u_h(y)=\Gamma(u_h) (y)$.
\end{lemma}

\begin{proof}
Recall that $\Gamma(u_h)=\widetilde{\Gamma}(u_h)$ on $\Conv (\mathcal{N}_h)$.
Since $p \in \chi_{\Gamma(u_h)} (x)$, using \eqref{g-u} and with 
$$
L_1(y) = \widetilde{\Gamma}(u_h)(x) + p \cdot (y-x),
$$ 
we have
$$
u_h(y) \geq \Gamma(u_h)(y) =\widetilde{\Gamma}(u_h)(y) \geq L_1(y), \ \forall y \in \mathcal{N}_h.
$$
Define
$$
a = \min \{ \, u_h(y) - L_1(y), y \in \mathcal{N}_h \, \}.
$$
We have $a\geq 0$. 
Assume that $a>0$ and consider the linear function
$$
L_2(z) =  \frac{a}{2} + \widetilde{\Gamma}(u_h)(x)  + p \cdot (z-x). 
$$ 
We claim that $u_h \geq L_2$ on $\mathcal{N}_h$ but $L_2(x) = a/2 + \widetilde{\Gamma}(u_h)(x) > \widetilde{\Gamma}(u_h)(x)$. This gives a contradiction since we should have $\widetilde{\Gamma}(u_h)(x) = \Gamma(u_h)(x) \geq L_2(x)$.
We have
$$
u_h(y) - L_1(y) \geq a, \ \forall y \in \mathcal{N}_h,
$$
and thus
\begin{align*}
u_h(y) &\geq a+ \widetilde{\Gamma}(u_h)(x) + p \cdot (y-x) \\
 & \geq \frac{a}{2} +  \widetilde{\Gamma}(u_h)(x)  + p \cdot (y-x) = L_2(y).
\end{align*}
We conclude that $a=0$ and there exists $y_0 \in \mathcal{N}_h$ such that $a=u_h(y_0)- L_1(y_0)=0$. Note that this implies that
$u_h(y_0) = \Gamma(u_h)(y_0)$, that is
\begin{equation} \label{a0}
u_h(y_0) - L_1(y_0) = \Gamma(u_h)(y_0) - L_1(y_0)=0. 
\end{equation}
We now show that $p \in \chi_{\Gamma(u_h)}(y_0)$. Let $z \in \R^d$. We have using
\eqref{a0} and $p \in  \chi_{\Gamma(u_h)} (x)$
\begin{align*}
\Gamma(u_h)(y_0) + p \cdot (z-y_0) & =  \Gamma(u_h)(y_0) + p \cdot (z-x) +  p \cdot (x-y_0)  \\
& =  \Gamma(u_h)(y_0) + \Gamma(u_h)(x) + p \cdot (z-x) \\
 & \qquad \qquad \qquad \qquad \qquad \qquad \qquad -  \Gamma(u_h)(x) - p \cdot (y_0-x)  \\
& =   \Gamma(u_h)(y_0) + L_1(z) - L_1(y_0) = L_1(z) \leq \widetilde{\Gamma}(u_h)(z),
\end{align*}
which shows that $p \in \chi_{\Gamma(u_h)}(y_0)$. \Qed

\end{proof}

\begin{lemma} \label{equivalence}
Let $u_h \in \mathcal{U}_h$
\begin{enumerate}
\item If for $x \in \Omega_h$, $\Gamma(u_h)(x) = u_h(x)$, then $\partial \Gamma(u_h)(x) =  \partial_h u_h(x)$

\item If $x \in \Omega_h$ and $\Gamma(u_h)(x) \neq u_h(x)$, then $ \partial_h u_h(x)=\emptyset$ and for any $p \in \chi_{\Gamma(u_h)}(x)$ there exists $y \in \mathcal{N}_h, y \neq x$ such that $p \in \chi_{\Gamma(u_h)}(x) \cap \chi_{\Gamma(u_h)}(y)$

\item If $x \in (\Conv (\mathcal{N}_h))^{\circ}$ but $x \notin \mathbb{Z}^d_h $, then for $p \in  \chi_{\Gamma(u_h)} (x)$ there exists $y \in \mathcal{N}_h, y \neq x$ such that $p \in \chi_{\Gamma(u_h)} (x) \cap \chi_{\Gamma(u_h)}(y)$.

\end{enumerate}

\end{lemma}

\begin{proof}

We start with the proof of the first statement.  Assume that $x \in \Omega_h$ and $\Gamma(u_h)(x) = u_h(x)$. 

\noindent
Let $p \in  \partial_h u_h(x)$. We have $u_h(y) \geq u_h(x) + p \cdot (y-x) \, \forall  y \in \mathcal{N}_h$. Consider the affine function $L(z) = u_h(x) + p \cdot (z-x)$. Since $L(y) \leq u_h(y)$ for all $y \in \mathcal{N}_h$, we have for all $z \in \Conv (\mathcal{N}_h)$
$$
\Gamma(u_h)(z) \geq L(z) = u_h(x) + p \cdot (z-x) = \Gamma(u_h)(x) + p \cdot (z-x),
$$
i.e. $p \in \partial \Gamma(u_h)(x)$. Thus $ \partial_h u_h(x) \subset  \partial \Gamma(u_h)(x)$. 

\noindent
Conversely, if $p \in \partial \Gamma(u_h)(x)$, as $\Gamma(u_h)$ is continuous on $\Conv (\mathcal{N}_h)$,
we have for all $y \in \mathcal{N}_h$, $\Gamma(u_h)(y) \geq u_h(x) + p \cdot (y-x)$. 
By \eqref{g-u}, this gives $p \in  \partial_h u_h(x)$ and completes the proof of the first claim.

Next, we prove the second statement. We first show that if $x \in \Omega_h$ and $ \partial_h u_h(x)\neq\emptyset$, then $\Gamma(u_h)(x) = u_h(x)$. 

\noindent
Let $p \in  \partial_h u_h(x)$. 
Since $u_h(y) \geq L(y) = u_h(x) + p \cdot (y-x)$ for all $y \in \mathcal{N}_h$, we obtain $\Gamma(u_h)(x) \geq L(x) = u_h(x)$ which by \eqref{g-u} gives $\Gamma(u_h)(x) = u_h(x)$.

\noindent
To conclude the proof of the second statement, assume that $x\in \Omega_h$ and $\Gamma(u_h)(x) \neq u_h(x)$. 
By Lemma \ref{equi0} there exists $y \in \mathcal{N}_h$ such that $p \in  \chi_{\Gamma(u_h)}(x)  \cap \chi_{\Gamma(u_h)} (y)$ and $u_h(y)=\Gamma(u_h) (y)$. Thus $y \neq x$ since $\Gamma(u_h)(x) \neq u_h(x)$.

Finally the third statement follows from Lemma \ref{equi0} since for $x \in (\Conv (\mathcal{N}_h))^{\circ}$ and $x \notin \mathbb{Z}^d_h$ we have $x \notin \mathcal{N}_h$.  \Qed

\end{proof}

\begin{remark} If we define a mesh function $u_h$ to be nodal convex at $x \in \Omega_h$ when $ \partial_h u_h(x)\neq\emptyset$, then for a nodal convex mesh function at $x \in \Omega_h$, we have $\Gamma(u_h)(x) = u_h(x)$ by Lemma \ref{equivalence} (2). It is also proven in \cite[Lemma 2.1]{Nochetto19} that for a mesh function $u_h$ which is nodal convex at all $x \in \Omega_h$, $\partial \Gamma(u_h)(x) =  \partial_h u_h(x)$ for all $x \in \Omega_h$.

\end{remark}

For a subset $E \subset \Omega$, we define 
$$
\partial_h u_h(E) = \cup_{x \in E \cap  \mathbb{Z}^d_h} \partial_h u_h(x),
$$
and 
define the discrete $R$-curvature of $u_h$ as the set function
$$
\omega(R,u_h,E) = \int_{\partial_h u_h(E)} R(p) \ud p.
$$

\begin{theorem}  \label{equi-measures}
We have for $u_h \in \mathcal{U}_h$ and a subset $E \subset (\Conv (N_h))^\circ \subset \Omega$, $\partial_h u_h(E) = \partial \Gamma(u_h)(E)$ up to a set of measure 0 and thus
$$
\omega(R,u_h,E) = \omega(R,\Gamma(u_h),E).
$$
\end{theorem}

\begin{proof} It is enough to show that for $E \subset (\Conv (N_h))^\circ \subset \Omega$, up to a set of measure 0, we have
$\partial_h u_h(E) = \partial \Gamma(u_h)(E)$. By Lemma \ref{equivalence}, for $x \in \Omega_h$, either $\partial_h u_h(x)$ is empty or $\partial_h u_h(x)=\partial \Gamma(u_h)(x)$. Thus $\partial_h u_h(E) \subset  \chi_{\Gamma(u_h)}(E)$.

For the reverse inclusion, let $x \in E$. If $x \in \Omega_h$ and $\Gamma(u_h)(x)=u_h(x)$, by Lemma \ref{equivalence}, 
$\partial \Gamma(u_h)(x) = \partial_h u_h(x) \subset \partial_h u_h(E)$. If $x \in \Omega_h$ and $\Gamma(u_h)(x)\neq u_h(x)$ or $x \in (\Conv (N_h))^\circ \setminus \Omega_h$, by Lemma \ref{equivalence} 
$$
\partial \Gamma(u_h)(x) \subset \{ \, p \in \R^d, p \in  \chi_{\Gamma(u_h)}(y) \cap  \chi_{\Gamma(u_h)}(z), y, z \in \Conv (N_h), y \neq z\, \}.
$$ 
By Lemma \cite[1.1.12]{Guti'errez2001} (applied on a bounded domain $\widetilde{\Omega}$ such that $\Omega \subset U \subset \widetilde{\Omega}$ for an open set $U$), 
the set on the right of the above inclusion is contained in a set of measure 0. 
This completes the proof. \Qed

\end{proof}


As in \cite{Nochetto19}, we define the contact set of the mesh function $u_h$ as 
$$
C_h^-(u_h) = \{ \, x \in \Omega_h, \Gamma(u_h)(x) = u_h(x) \, \}.
$$

By Theorem \ref{equi-measures}, $\cup_{x \in \Omega_h} \partial_h u_h(x) = \cup_{x \in C_h^-(u_h)} \partial \Gamma(u_h)(x)$ up to a set of measure 0. By Lemma \cite[1.1.12]{Guti'errez2001} we have for $x \neq y$, $|\partial \Gamma(u_h)(x) \cap \partial \Gamma(u_h)(y)|=0$. It thus follows from Lemma \ref{inter-meas} that $ \omega(1,u_h,\Omega_h) =  \sum_{x \in C_h^-(u_h)}  | \partial \Gamma(u_h)(x) |$ and thus by Lemma \ref{equivalence} 
\begin{equation} \label{mass-gamma}
 \omega(1,u_h,\Omega_h) =  \sum_{x \in C_h^-(u_h)}  | \partial \Gamma(u_h)(x) | = \sum_{x \in C_h^-(u_h)}  |\partial_h u_h(x)| = \sum_{x \in \Omega_h}  |\partial_h u_h(x)|. 
\end{equation}

The discrete Aleksandrov-Bakelman-Pucci maximum principle \cite{Nochetto14} can be stated as follows.

\begin{lemma} \label{d-stab}
Let $u_h \in \mathcal{U}_h$ such that $u_h \geq 0$ on $\partial \Omega_h$. Then for 
$x \in \Omega_h$ 
$$
 u_h(x) \geq - C(d) \bigg[ \diam(\Omega)^{d-1} d(x,\partial \Omega)  \omega(1,u_h,\Omega_h) \bigg]^{\frac{1}{d}},
$$
for a positive constant $C(d)$ which depends only on $d$. 
\end{lemma}

Note that $\omega(R,u_h,E)=\omega(R,u_h,\{ \, x \, \})$ for $|E|$ sufficiently small and $x \in E$. We will make the abuse of notation
$$
\omega(R,u_h, x )  = \omega(R,u_h,\{ \, x \, \}).
$$

\begin{definition} \label{mass}
We refer to $\omega(1,v_h,\Omega_h)$ as  Monge-Amp\`ere mass of the mesh function $v_h$.
\end{definition}

It can be shown that the maximum of a discrete convex convex function occurs on $\partial \Omega_h$. Mesh functions with Monge-Amp\`ere masses uniformly bounded have a uniform lower bound by Lemma \ref{d-stab}. Therefore discrete convex mesh functions with uniformly bounded Monge-Amp\`ere masses are uniformly bounded when their boundary values are uniformly bounded.

We now give a geometric proof of a refined version of Lemma \ref{equivalence} (2)--(3) based on the observation that the graph of the convex envelope $\Gamma(u_h)$ coincides on $\Conv (\mathcal{N}_h)$ with the lower part of the convex polyhedron which is the convex hull $S$ of the points $(x, u_h(x)) \in \R^{d+1}, x \in \mathcal{N}_h$.
It follows immediately that $\Gamma(u_h)$ is continuous on $\Conv (\mathcal{N}_h)$.

The lower faces of $S$ form the graph of the convex envelope of $u_h$ over $\Conv (\mathcal{N}_h)$. 
We recall that a lower face $F$ is a face such that for $x \in F$, $x-(0,\ldots0,\lambda) \notin S$ for all $\lambda>0$.
The projections onto $\R^d$ of the vertices of the lower faces of $S$ form a convex decomposition of
$\Conv(\mathcal{N}_h)$. Each element of the decomposition is a convex hull of projections onto $\R^d$ of vertices of lower faces of $S$. Each vertex of the decomposition is a point $x \in \mathcal{N}_h$ where $\Gamma(u_h)(x) = u_h(x)$. Each element $C$ of the decomposition can be further subdivided into simplices with vertices among the vertices of $C$. This results in a triangulation of $\Conv(\mathcal{N}_h)$ on which $\Gamma(u_h)$ is piecewise linear. We will refer to such a triangulation as an induced triangulation.


Let $\mathcal{T}(\Gamma(u_h))$ denote a triangulation induced on $\Conv (\mathcal{N}_h)$ by $\Gamma(u_h)$. We will need below the following theorem

\begin{theorem} \cite[Theorem 3]{ioffe2009theory} \label{extreme}
Let $S$ be a finite set and $f_s$ be convex for each fixed $s \in S$ on $X\subset \R^d$ and continuous in some neighborhood of $x_0 \in X$. Let
$
f(x) = \sup \{ \, f_s(x), s \in S \, \}
$ and define $S(x) =  \{ \,s \in S, f_s(x)=f(x)  \, \}$. 
Then
$$
\partial f(x_0) = \tir{\Conv \bigg( \cup_{s \in S(x_0)} \partial f_s(x_0)} \bigg). 
$$

\end{theorem}

\begin{lemma}
We have
for $x \in \R^d$
\begin{equation} \label{induced-x}
\widetilde{\Gamma}(u_h)(x) = \max \{ \, L(x), L=\Gamma(u_h)|_T, T \in \mathcal{T}(\Gamma(u_h)) \, \}.
\end{equation}

\end{lemma}

\begin{proof}
The left hand side of \eqref{induced-x} defines a convex function $F(x)$ on $\R^d$. By construction we have
$$
\widetilde{\Gamma}(u_h)(x) = \Gamma(u_h)(x) = F(x), x \in \Conv (\mathcal{N}_h).
$$
By the definition \eqref{extension} of convex extension, we have
\begin{equation} \label{ind-01}
F(x) \leq \widetilde{\Gamma}(u_h)(x) \ \forall x \in \R^d.
\end{equation}
Let $y \in (\Conv (\mathcal{N}_h))^\circ$ and $p \in \partial \Gamma(u_h)(y)$. 
Recall that if $Z \subset \R^d$ is compact, $\Conv(Z)$ is closed.  
By Theorem  \ref{extreme}, $p$ is in the convex hull of the constant gradients $D \Gamma(u_h)|_T$ for $T \in  \mathcal{T}(\Gamma(u_h))$ containing $y$. That is, we can find elements $T_i, i=1,\ldots,N$,
$T_i \in  \mathcal{T}(\Gamma(u_h))$ for all $i$ and $\lambda_i \in [0,1], i=1,\ldots,N$ such that
$\sum_{i=1}^N \lambda_i =1$ and
$$
p=\sum_{i=1}^N \lambda_i D \Gamma(u_h)|_{T_i}.
$$
We have for $x \in \R^d$
\begin{align*}
\Gamma(u_h)(y) + p \cdot (x-y) & = \Gamma(u_h)(y) + \sum_{i=1}^N \lambda_i D \Gamma(u_h)|_{T_i} \cdot (x-y)  \\
& = \sum_{i=1}^N \lambda_i \bigg(\Gamma(u_h)(y) + D \Gamma(u_h)|_{T_i} \cdot (x-y) \bigg) \\
& \leq \sum_{i=1}^N \lambda_i  F(x) = F(x),
\end{align*}
where we observed that $\Gamma(u_h)(y) + D \Gamma(u_h)|_{T_i} \cdot (x-y)$ is the restriction to $T_i$ of $\Gamma(u_h)$.
By \eqref{ind-01} and the definition of convex extension, we obtain $F(x) \leq \widetilde{\Gamma}(u_h)(x) \leq F(x) \ \forall x \in \R^d$.

\Qed
\end{proof}

\begin{theorem} \label{geo-proof}
$\forall x\in (\Conv (N_h))^\circ$ and $p \in  \partial \Gamma(u_h)(x)$,  $\exists y \in \mathcal{N}_h$ such that $p \in   \chi_{\Gamma(u_h)}(x) \cap \chi_{\Gamma(u_h)}(y)$ and $\Gamma(u_h)(y) = u_h(y)$.  
Thus $\partial \Gamma(u_h) ((\Conv (N_h))^\circ) \subset \chi_{\Gamma(u_h)}(\mathcal{N}_h)$.
If $x \in \Omega_h$ and $\Gamma(u_h)(x) \neq u_h(x)$, we can choose $y \in \mathcal{N}_h$ such that $y \neq x$. 
Moreover
 $\chi_{\Gamma(u_h)}(\R^d) \subset \Conv \big( \chi_{\Gamma(u_h)}(\mathcal{N}_h) \big)$.

\end{theorem}

\begin{proof}


Let $x \in \Omega_h$ such that $\Gamma(u_h)(x) \neq u_h(x)$. 
Let $y$ be a vertex of an element $T$ of an induced triangulation such that $x \in T$. By construction $y \in  \mathcal{N}_h$. Since $\Gamma(u_h)(x) \neq u_h(x)$ and $x \in \Omega_h$, $x$ cannot be a vertex of $T$. Thus $x \neq y$. If $x \in T^\circ$ then for $p \in \partial \Gamma(u_h)(x)$, $p$ is the gradient of the linear function on $T$ which is the restriction on $T$ of the convex envelope. We conclude that $p \in \partial \Gamma(u_h)(x) \cap \chi_{\Gamma(u_h)}(y)$. If $x \in \partial T$ and $x$ is not a vertex of $T$,
we have 
$$
\Gamma(u_h)(x) = \max \{\,  L(x), L=\Gamma(u_h)|_T, T \in \omega(x) \, \},
$$
where $\omega(x)$ is the union of the elements of $\mathcal{T}(\Gamma(u_h))$ which contain $x$. But
$\omega(x) \subset \omega(y)$ for a vertex $y$. Thus $p \in \partial \Gamma(u_h)(x) \cap \chi_{\Gamma(u_h)}(y)$ as well using Theorem \ref{extreme}.


If  $x \in \Omega_h$, we can choose $y=x$. 
If $x \in (\Conv (N_h))^\circ$ and $x \notin \Omega_h$, the same argument gives the existence of a vertex $y \in  \mathcal{N}_h$ such that for $p \in \partial \Gamma(u_h)(x)$, $p \in \partial \Gamma(u_h) (x) \cap \chi_{\Gamma(u_h)}(y)$. 
We conclude that
\begin{equation*}
\partial \Gamma(u_h) ((\Conv (N_h))^\circ) \subset   \chi_{\Gamma(u_h)}(\mathcal{N}_h). 
\end{equation*}

Using the representation \eqref{induced-x} and Theorem \ref{extreme}, we have $$ \chi_{\Gamma(u_h)}(\mathcal{N}_h) \subset \Conv \{ \, D  \Gamma(u_h)|_{T}, T \in \mathcal{T}(\Gamma(u_h)) \, \}.$$

Therefore $\partial \Gamma(u_h) ((\Conv (N_h))^\circ)$ is bounded and by Lemma \ref{ext-subd}
$$
\chi_{\Gamma(u_h)}(\R^d) \subset \Conv \bigg( \tir{\partial \Gamma(u_h) ((\Conv (N_h))^\circ)} \bigg) \subset \Conv \big(  \tir{ \chi_{\Gamma(u_h)}(\mathcal{N}_h)} \big). 
$$
But for each $x \in \Omega_h$, $\chi_{\Gamma(u_h)}(x)$ is a closed set and thus $\chi_{\Gamma(u_h)}(\mathcal{N}_h)$ is closed as a finite union of closed sets and hence compact. Therefore 
$
\chi_{\Gamma(u_h)}(\R^d) \subset \Conv \big( \chi_{\Gamma(u_h)}(\mathcal{N}_h) \big)$. This completes the proof.



\Qed

\end{proof}


\begin{remark} \label{geo-cvx-envelope}
Let $x \in \Conv (\Omega_h)$ and $T$ an element of an induced triangulation $\mathcal{T}(\Gamma(u_h))$ such that $x \in T$ and $\Gamma(u_h)$ is linear on $T$. Then if $x_i, i=1,\ldots,d$ denote the vertices of $T$, we have $x=\sum_{i=1}^{d+1} \lambda_i x_i$, $ \lambda_i \geq 0$, $ \sum_{i=1}^{d+1} \lambda_i=1$. 
Moreover $\Gamma(u_h)(x_i)=u_h(x_i)$ for all $i$. 
We have $\Gamma(u_h)(x)=\sum_{i=1}^{d+1} \lambda_i \Gamma(u_h)(x_i)$ since $\Gamma(u_h)$ is linear on $T$.

\end{remark}

\section{Convergence of mesh functions and their convex envelopes} \label{interplay}

In this section we study the connection between the convergence of mesh functions and the convergence of their convex envelopes.
\begin{definition} \label{def-cv-mesh}
Let $u_h \in \mathcal{U}_h$ for each $h >0$. We say that $u_h$ converges to a function $u$ uniformly on $\tir{\Omega}$ if and only if for each sequence $h_k \to 0$ and for all $\epsilon >0$, there exists $h_{-1} >0$ such that for all $h_k$, $0< h_k < h_{-1}$, we have
$$
\max_{x \in \mathcal{N}_{h_k}} |u_{h_k}(x) - u(x)| < \epsilon.
$$
\end{definition}

We note that if $u_{h_k}$ converges uniformly on $\tir{\Omega}$ to a bounded function $u$, then $\Gamma(u_{h_k})$ is locally uniformly Lipschitz for $h_k$ sufficiently small.  

Put $m=\min \{\, u(x), x \in \tir{\Omega} \, \}$ and $M=\max \{\, u(x), x \in \tir{\Omega} \, \}$. For $h_k$ sufficiently small, we have
$m-1 \leq u_{h_k}(x) \leq M+1$ for all $x \in \mathcal{N}_{h_k}$. As $m-1$ is convex we get $\Gamma(u_{h_k}) \geq m-1$. If $L$ is affine and $L(x) \leq M+1$ for all $x \in \mathcal{N}_{h_k}$, then $L(x) \leq M+1$ for all $x\in \R^d$. This gives
$\Gamma(u_{h_k}) \leq M+1$. Thus $\Gamma(u_{h_k})$ is uniformly bounded. Since $\Gamma(u_{h_k})$ is convex, $\Gamma(u_{h_k})$ is locally uniformly Lipschitz, c.f. for example \cite[Theorem C]{Roberts}.

\begin{lemma} \label{cvg-cvx}
Assume that $u_h$ converges to a convex function $u$ uniformly on $\tir{\Omega}$. 
Assume also that $u$ is bounded. 
Then 
$\Gamma(u_h)$ converges to $u$
uniformly on compact subsets of $\Omega$. 
\end{lemma}

\begin{proof} Let $\epsilon >0$ and $h_{-1} >0$ such that for all $h_k$, $0< h_k < h_{-1}$ and $x \in \mathcal{N}_{h_k}$
\begin{equation} \label{u-1}
u_{h_k}(x) - \frac{\epsilon}{2} < u(x) <  u_{h_k}(x) + \frac{\epsilon}{2}.
\end{equation}
From \eqref{g-u}, we get
$$
\Gamma(u_{h_k})(x) - \frac{\epsilon}{2} \leq u_{h_k}(x) - \frac{\epsilon}{2} < u(x), x \in \mathcal{N}_{h_k}.
$$
But $\widetilde{\Gamma}(u_{h_k}) - \epsilon/2$ is a convex function. Thus, by definition of the discrete convex envelope \eqref{disc-cv}, we obtain
\begin{equation} \label{u-2}
\Gamma(u_{h_k})(x)  - \frac{\epsilon}{2}  \leq \Gamma_h(u)(x), x \in \mathcal{N}_{h_k}.
\end{equation}

Recall from Remark \ref{geo-cvx-envelope} that for every $x \in \Conv(\Omega_h)$, we can find $x_i \in \mathcal{N}_{h_k}, i=1,\ldots, d+1$ such that $u_{h_k}(x_i) = \Gamma(u_{h_k})(x_i)$ for all $i$ and  $x=\sum_{i=1}^{d+1} \lambda_i x_i$, $ \lambda_i \geq 0$, $ \sum_{i=1}^{d+1} \lambda_i=1$. Moreover, $\Gamma(u_{h_k})$ is linear on the convex hull of $\{ x_i, i=1,\ldots, d+1 \, \}$. Thus
$$
\Gamma(u_{h_k})(x) = \sum_{i=1}^{d+1} \lambda_i \Gamma(u_{h_k})(x_i) = \sum_{i=1}^{d+1} \lambda_i u_{h_k}(x_i). 
$$
It follows from \eqref{u-1} and the convexity of $u$ that 
\begin{equation} \label{u-3}
u(x) \leq \sum_{i=1}^{d+1} \lambda_i  u(x_i) < \sum_{i=1}^{d+1} \lambda_i  u_{h_k}(x_i) + \frac{\epsilon}{2} = \Gamma(u_{h_k})(x) + \frac{\epsilon}{2}.
\end{equation}
By Theorem \ref{discrete-cvx-envelope-thm}, since $u$ is convex, for $x \in \mathcal{N}_{h_k}$, $u(x) = \Gamma(u)(x) =  \Gamma_h(u)(x)$. We conclude from \eqref{u-2} and \eqref{u-3} that
$$
\Gamma(u_{h_k})(x)  - \frac{\epsilon}{2}  \leq u(x) < \Gamma(u_{h_k})(x) + \frac{\epsilon}{2}.
$$
Thus $u-\Gamma(u_{h_k})$ converge uniformly to 0 on $\tir{\Omega}$, as mesh functions, i.e. in the sense of Definition \ref{def-cv-mesh}. We next prove that $u-\Gamma(u_{h_k})$ converges uniformly to 0 on compact subsets of $\Omega$ as functions on $\tir{\Omega}$. 

Let $K$ be a compact subset of $\Omega$. We may assume that $K \subset (\Conv(\mathcal{N}_{h_k}   ))^\circ$. By Lemma \ref{continuity-convex}, $u$ is continuous on $K$ and hence uniformly continuous on $K$. Given $\delta >0$, there exists $\alpha >0$ such that  
$$
|y-x| < \alpha \implies |u(y) - u(x)| < \frac{\delta}{4}, x, y \in K.
$$
Assume now that the convergence is not uniform on $K$. We may assume that the sequence $h_k$ is chosen such that for all $x \in \Conv(\mathcal{N}_{h_k})$, one can find $y \in \mathcal{N}_{h_k}$ such that $|y-x| < \alpha$. By assumption, there exists $\delta >0$ such that for all $k$, there exists $x_k \in K$ 
with $|u(x_k)-\Gamma(u_{h_k})(x_k)|> \delta$. Let $y_k \in \mathcal{N}_{h_k} \cap K$ such that $|y_k - x_k| < \alpha$. By the local equicontinuity of $\Gamma(u_{h_k})$, we may assume that 
$$
|\Gamma(u_{h_k}) (x_k) - \Gamma(u_{h_k})(y_k) | \leq  \frac{\delta}{4}.
$$ 

Put $a=u(y_k) -\Gamma(u_{h_k}) (y_k)$, $b=u(x_k) - u(y_k)$ and $c=\Gamma(u_{h_k})  (y_k) - \Gamma(u_{h_k}) (x_k)$. Since $|a| \geq |a+b+c| - |b| - |c|$, we get
\begin{multline*}
|u(y_k) -  \Gamma(u_{h_k}) (y_k)| \geq |u(x_k)  -  \Gamma(u_{h_k}) (x_k)| - |u(x_k) - u(y_k)| \\
 - |\Gamma(u_{h_k})  (y_k) - \Gamma(u_{h_k}) (x_k)| \geq \delta -  \frac{\delta}{4} -  \frac{\delta}{4} =  \frac{\delta}{2}. 
\end{multline*}
A contradiction. This completes the proof. \Qed
\end{proof}

\begin{theorem} \label{main0}
Assume that $u_h$ converges to a convex function $u$ uniformly on $\tir{\Omega}$. 
Assume also that $u$ is bounded. 
Then
$\omega(R,u_h,.)$ weakly converges to $\omega(R,u,.)$.
\end{theorem}

\begin{proof} 
 From Lemma \ref{cvg-cvx}, $\Gamma(u_h)$ converges to $u$ uniformly on compact subsets of $\Omega$. The result then follows from Theorems \ref{equi-measures} and \ref{weak-cvg}.
\Qed
\end{proof}

If $u_h$ only converges uniformly on compact subsets, we give below conditions under which $\omega(R,u_h,.)$ weakly converges. In some situations discussed in section \ref{compact} one deals with sequences for which both $u_{h}$ and $\Gamma(u_h)$ converge uniformly on compact subsets. We recall that for a family of sets $A_k$
$$
\limsup_k A_k  = \cap_{n} \cup_{k \geq n} A_k \text{ and }  \liminf_k A_k  = \cup_{n} \cap_{k \geq n} A_k.
$$

For Lemmas \ref{lem-weak01} and \ref{lem-weak02} below, we consider a compact $K \subset \Omega$ and an open set $U$ such that $K \subset U  \subset \tir{U} \subset \Omega$. Let $\Gamma_U(u_h)$ denote the convex envelope of $u_h$ on $\tir{U}$. If $u_h \to u$ on $ \tir{U}$,
by  Lemma \ref{cvg-cvx},  $\Gamma_U(u_h)$ converges uniformly to $u$ on $K$. We have  
\begin{equation} \label{gamma-U}
 \Gamma( u_{h_k}) \leq \Gamma_U( u_{h_k}) \text{ on } U \text{ and on } U \cap \mathcal{N}_{h_k}, \Gamma( u_{h_k}) \leq \Gamma_U( u_{h_k}) \leq u_{h_k}.
 \end{equation}

\begin{lemma} \label{lem-weak01}
Assume that $u_h \to u$ uniformly on compact subsets of $\Omega$, with $u$ convex and continuous. Then for $K \subset \Omega$ compact and any sequence $h_k \to 0$, up to a set of measure zero
$$
\limsup_{h_k \to 0} \partial \Gamma( u_{h_k})(K) \subset \partial u(K).
$$
\end{lemma}

\begin{proof}

Let $U$ be open such that $K \subset U  \subset \tir{U} \subset \Omega$.  By \cite[Lemma 1.2.2]{Guti'errez2001}
$$
\limsup_{h_k \to 0} \partial \Gamma_U( u_{h_k})(K) \subset \partial u(K).
$$
 Let $x \in K$ and $p \in \partial \Gamma( u_{h_k})(x)$. 
If $\Gamma( u_{h_k})(x) = \Gamma_U( u_{h_k})(x)$, using \eqref{gamma-U} we obtain
$p \in  \partial \Gamma_U( u_{h_k})(x) \subset \partial \Gamma_U( u_{h_k})(K)$. 

Assume that $\Gamma( u_{h_k})(x) \neq \Gamma_U( u_{h_k})(x)$. 
By Theorem \ref{geo-proof}, there exists $y \in \mathcal{N}_h$ such that $p \in 
\partial \Gamma( u_{h_k})(x) \cap \partial \chi_{\Gamma( u_{h_k})}(y)$ and $\Gamma(u_{h_k})(y) = u_{h_k}(y)$. If $y \in K$, this implies by \eqref{gamma-U} that $\Gamma( u_{h_k})(y) = \Gamma_U( u_{h_k})(y)$ and hence $p \in  \partial \Gamma_U( u_{h_k})(K)$.  If $y \notin K$, 
$x \neq y$ and $p$ is contained in a set of measure 0. Thus, up to a set of measure zero, 
$\partial \Gamma ( u_{h_k})(K)\subset \partial \Gamma_U( u_{h_k})(K)$. The result then follows.

\Qed
\end{proof}

Recall from Lemma \ref{extension-convex} that the extension $\tilde{u}$ of $u$ given by \eqref{extension} satisfies $\tilde{u}=u$ on $\Omega$ and  that for all $x \in \Omega$, $\chi_u(x) = \partial u(x)$.

\begin{lemma} \label{lem-weak02}
Assume that $u_h \to u$ uniformly on compact subsets of $\Omega$, with $u$ convex and continuous. Assume that $K$ is compact and $U$ is open with $K \subset U  \subset \tir{U} \subset \Omega$ and that for any sequence $h_k \to 0$, a subsequence $k_j$ 
and $z_{k_j} \in \mathcal{N}_{h_{k_j}}$ with $z_{k_j} \to z_0 \in \partial \Omega$, we have
\begin{equation} \label{bc-condition}
\liminf_{j \to \infty} \tilde{u}(z_{k_j}) \leq \limsup_{j \to \infty} u_{h_{k_j}} (z_{k_j}).
\end{equation}
Then, up to a set of measure zero, for any sequence $h_k \to 0$
$$
 \partial u(K) \subset \liminf_{h_k \to 0}\partial \Gamma(u_{h_k})(U).
$$

\end{lemma}

\begin{proof} The proof is analogous to the proof of \cite[Lemma 3.3]{GutierrezNguyen07}.  By \cite[Lemma 1.1.3]{Guti'errez2001}, $\partial u(\tir{U})$ is bounded, and
 since $u$ is bounded on $\tir{U}$ and $\Omega$ bounded, $\tilde{u}$ is finite on $\tir{U}$. Define
$$
W = \{ \, p \in \R^d, p \in \partial \tilde{u}(x_1) \cap \partial \tilde{u}(x_2), \, \text{for some} \, x_1, x_2 \in \R^d, x_1 \neq x_2 \, \}.
$$
By \cite[Lemma 1.1.12]{Guti'errez2001}, $|W|=0$. Let $p \in \partial v(K) \setminus W$. By definition of $W$, 
there exists a unique $x_0 \in K$ such that $p \in \partial \tilde{u}(x_0)$  and for all $x \in \R^d, x \neq x_0$ we have $p \notin \partial \tilde{u}(x)$. We claim that
\begin{equation} \label{part-ineq0}
\tilde{u}(x) > \tilde{u}(x_0) + p \cdot (x-x_0), x \in \R^d, x \neq x_0,
\end{equation}
that is, $\tilde{u}(x_0) + p \cdot (x-x_0)$ is a strictly supporting hyperplane to the graph of $\tilde{u}$ at $x_0$. Otherwise, the plane would touch the graph at another point $x_1$ and would be a support at $x_1$ as well, contradicting the assumption that $p \notin W$. We refer to
\cite[p. 7]{Guti'errez2001} for an analytical proof. 

By \cite[Lemma 1.2.2]{Guti'errez2001}
$$
 \partial u(K) \setminus W \subset \liminf_{h_k \to 0}\partial \Gamma_U(u_{h_k})(U).
$$
This means that there exists $k_0$ such that for all $ k \geq k_0 $, one can find $x_k \in U$ and $p \in \partial \Gamma_U(u_{h_k})(x_k)$. Thus
\begin{equation} \label{part-ineq1}
 u_{h_k}(x) \geq \Gamma_U(u_{h_k})(x) \geq  \Gamma_U(u_{h_k})(x_k) + p \cdot (x-x_k), \forall x \in \tir{U} \cap \mathcal{N}_{h_k},
 \end{equation}
where we used the continuity of $\Gamma_U(u_{h_k})$ and \eqref{gamma-U}. We now claim that \eqref{part-ineq1} actually holds for all $x \in \mathcal{N}_{h_k}$ when $k \geq k_0$. Otherwise one can find a subsequence $k_j$ and $z_{k_j} \in (\tir{\Omega} \setminus \tir{U}) \cap \mathcal{N}_{h_{k_j}}$ such that
\begin{equation} \label{part-ineq2}
u_{h_{k_j}}(z_{k_j}) < \Gamma_U(u_{h_{k_j}})(x_{k_j}) + p \cdot (z_{k_j}-x_{k_j}).
\end{equation}
Since $\Omega$ is bounded, up to a subsequence, we may assume that $z_{k_j} \to z_0 \in \tir{\Omega} \setminus U$. We show that
\begin{equation} \label{part-ineq3}
\tilde{u}(z_0) \leq  \tilde{u}(x_0) + p \cdot (z_0-x_0).
\end{equation} 

{\it Case 1}: $z_0 \in \Omega \setminus U$. Let $Q$ be a compact subset such that $\tir{U} \subset Q$ and $z_{k_j} \in Q$ for $j$ sufficiently large. 
Using the uniform convergence of $u_h$ to $u$ on $Q$, the uniform continuity of $u$ on $Q$, the uniform convergence of $\Gamma_Q(u_h)$ to $u$ on $Q$, and taking limits in \eqref{part-ineq2}, we obtain 
$u(z_0) \leq u(x_0) + p \cdot (z_0-x_0) $. Since both $x_0$ and $z_0$ are in $\Omega$, this gives \eqref{part-ineq3}.

{\it Case 2}: $z_0 \in \partial \Omega \setminus U$. Now we have 
$$
\limsup_{j \to \infty} u_{h_{k_j}}(z_{k_j}) \leq \tilde{u}(x_0) + p \cdot (z_0-x_0),
$$
where we used the uniform convergence of $\Gamma_U(u_h)$ to $u$ on $\tir{U}$.
Note that $\tilde{u}$ is lower semi-continuous as the supremum of affine functions. Thus, using the assumption \eqref{bc-condition}, we obtain
\begin{align*}
\limsup_{j \to \infty} u_{h_{k_j}}(z_{k_j}) \geq  \liminf_{j \to \infty} \tilde{u}_{}(z_{k_j}) \geq \tilde{u}(z_0).
\end{align*}
Hence \eqref{part-ineq3} also holds in this case.

Finally we  note that \eqref{part-ineq3} contradicts \eqref{part-ineq0} and therefore \eqref{part-ineq2} cannot hold, i.e. \eqref{part-ineq1} actually holds for all $x \in \mathcal{N}_{h_k}$ when $k \geq k_0$. But this implies that $\Gamma(u_{h_k})(x) \geq \Gamma_U(u_{h_k})(x_k) + p \cdot (x-x_k), \forall x \in \Conv (\mathcal{N}_{h_k})$ and $k \geq k_0$. In particular, for $x=x_k$, we have $\Gamma(u_{h_k})(x_k) \geq \Gamma_U(u_{h_k})(x_k) \geq \Gamma(u_{h_k})(x_k)$, where we used \eqref{gamma-U}. We conclude that $\Gamma(u_{h_k})(x) \geq \Gamma(u_{h_k})(x_k) + p \cdot (x-x_k), \forall x \in \Conv (\mathcal{N}_{h_k})$ and $k \geq k_0$, i.e. $p \in \cup_n \cap_{k \geq n} \partial \Gamma(u_{h_k})(U)$. This concludes the proof.

\Qed
\end{proof}


\begin{corollary} \label{bc-cor}
Assume that $u_h \to u$ uniformly on $\tir{\Omega}$, with $u$ convex and continuous on $\tir{\Omega}$. Assume that $K$ is compact and $U$ is open with $K \subset U  \subset \tir{U} \subset \Omega$. Then, up to a set of measure zero, for any sequence $h_k \to 0$
$$
 \partial u(K) \subset \liminf_{h_k \to 0} \partial \Gamma(u_{h_k})(U).
$$
\end{corollary}

\begin{proof}
By Lemma  \ref{lem-weak02}, it is enough to show that \eqref{bc-condition} holds. The proof follows from \cite[Remark 3.2]{GutierrezNguyen07}. Put $a_j=u_{h_{k_j}} (z_{k_j}) - u_{} (z_{k_j})$ and $b_j=u_{} (z_{k_j})$. Recall that
$\limsup_{j} (a_j+b_j) \geq \limsup_j a_j + \liminf_j b_j$. This gives
$$
\limsup_{j \to \infty} u_{h_{k_j}} (z_{k_j}) \geq \limsup_{j \to \infty} (u_{h_{k_j}} - u) (z_{k_j}) +
\liminf_{j \to \infty} u_{} (z_{k_j}) = \liminf_{j \to \infty} u_{} (z_{k_j}).
$$
Since $u \in C(\tir{\Omega})$, we have $\tilde{u}=u$ on $\tir{\Omega}$. This completes the proof.

\Qed
\end{proof}

It follows immediately from Corollary \ref{bc-cor}, Lemmas \ref{lem-weak01} and \ref{lem-weak02},  an equivalence criteria of weak convergence of measures, c.f. for example  \cite[Theorem 1, section 1.9]{Evans-Gariepy}, and Theorem \ref{equi-measures}, 
that if $u_h \to u$ uniformly on $\tir{\Omega}$, with $u$ convex and continuous on $\tir{\Omega}$, 
then $\omega(R,u_h,.) = \omega(R,\Gamma(u_h),.)$ tend to $\omega(R,u,.)$ weakly.

\begin{theorem} \label{mass-thm}
Assume that $u_h \to u$ uniformly on compact subsets of $\Omega$, with $u$ convex and continuous. We also assume that
$\omega(1,u_h,\Omega_h) \leq C$ for a constant $C$ independent of $h$, and $u_h=g$ on $\partial \Omega_h$ where $g  \in C(\tir{\Omega})$ is convex. Then \eqref{bc-condition} holds and as a consequence $\omega(R,u_h,.)$ tend to $\omega(R,u,.)$ weakly.
\end{theorem}

\begin{proof} We prove in Theorem \ref{main2} below, that under the assumptions of the theorem, $u \in C(\tir{\Omega})$ and $u=g$ on $\partial \Omega$. (Note that the discrete convexity assumption of $u_h$ and the uniform boundedness of $u_h$ in Theorem \ref{main2} only allows to have a subsequence which converges uniformly on compact subsets of $\Omega$ to a continuous convex function).

Thus, for $z_{k_j} \in \mathcal{N}_{h_{k_j}}$ with $z_{k_j} \to z_0 \in \partial \Omega$, we have $\liminf_{j \to \infty} u(z_{k_j}) = g(z)$. 
As in the proof of Theorem \ref{main2} and using the continuity of $g$ and $u_h=g$ on $\partial \Omega_h$
\begin{align*}
g(z) &\leq 
\begin{cases}  \limsup \{ \, u_{h_{k_j}} (z_{k_j}),  z_{k_j} \in \Omega_{h_{k_j}}\, \} \\
  \limsup \{ \, u_{h_{k_j}} (z_{k_j}),  z_{k_j} \in \partial \Omega_{h_{k_j}} \, \}. 
\end{cases}
\end{align*}
We conclude that $\liminf_{j \to \infty} u(z_{k_j}) = g(z) \leq \limsup_{j \to \infty} u_{h_{k_j}} (z_{k_j})$, i.e. \eqref{bc-condition} holds, 
as $u \in C(\tir{\Omega})$ and thus $\tilde{u}=u$ on $\tir{\Omega}$.

The weak convergence result then follows from Lemmas \ref{lem-weak01} and \ref{lem-weak02},  and an equivalence criteria of weak convergence of measures, c.f. for example  \cite[Theorem 1, section 1.9]{Evans-Gariepy}.

\Qed
\end{proof}

\section{Compactness of mesh functions with Monge-Amp\`ere masses uniformly bounded} \label{compact}

We use $C$ for a constant which may change at occurrences. In this section, we are interested in conditions under which a family of mesh functions has a subsequence which converges uniformly on $\tir{\Omega}$ or uniformly on compact subsets of $\Omega$. In both cases, we will also have uniform convergence on compact subsets of the convex envelopes.

We note that if $|u_h| \leq M$ for a constant $M$ independent of $h$, $|\Gamma(u_h)|\leq M$ on $\Conv(\mathcal{N}_h)$. See section \ref{interplay} for a similar argument. 

\begin{theorem} \label{mass-compact03}
Assume that $|u_h| \leq M$ for a constant $M$ independent of $h$. Then, there is a subsequence $h_k$ such that
 $\Gamma(u_{h_k})$ converges uniformly on compact subsets to a convex function $v$ on $\Omega$.

\end{theorem}

\begin{proof} Let $h_n=1/2^n$, a decreasing sequence such that $h_n \to 0$. 
Put $D_i = (\Conv (\Omega_{h_i}))^\circ$. For $i > j$, 
$D_j \subsetneq D_i$ and $\cup_{i=1}^\infty D_i = \Omega$. Since $\Gamma(u_h)$ is uniformly bounded and convex, for each $i$, $\Gamma(u_{h_j})$ is uniformly Lipschitz on  $\tir{D_i}$ for $j>i+1$. By the Arzela-Ascoli theorem, 
there is a subsequence $h_{n_i(k)}$ such that $\Gamma(u_{ h_{n_i(k)} })$ converges uniformly on $\tir{D_i}$ as $k \to \infty$. We may assume that $h_{n_i(k)}$ is a subsequence of $h_{n_{i-1}(k)}$. For the diagonal sequence $h_{n_k(k)}$, $\Gamma(u_{ h_{n_k(k)} })$ converges uniformly on compact subsets of $\Omega$ to a function which is convex on $\Omega$ as a uniform limit of convex functions.

\Qed
\end{proof}

The next lemma gives conditions under which there is a subsequence whose convex envelopes converge uniformly on $\tir{\Omega}$. The assumption that $\chi_{\Gamma(u_h)}(\mathcal{N}_h)$ is uniformly bounded can be verified for certain discretizations of the second boundary value problem for the Monge-Amp\`ere equation. 

Let $\widetilde{\Omega}$ be a rectangular domain and $U$ an open set 
such that 
$$
\tir{\Omega}  \subset U \subset \widetilde{\Omega}.
$$

\begin{theorem} \label{mass-compact02}
Assume that
$|u_h| \leq M$ for a constant $M$ independent of $h$ and $\chi_{\Gamma(u_h)}(\mathcal{N}_h)$ is uniformly bounded. Then, $\widetilde{\Gamma}(u_h)$ is uniformly bounded on $\tir{\Omega}$ and uniformly Lipschitz on $\tir{\Omega}$. As a consequence, there is a subsequence $h_k$ such that
 $\Gamma(u_{h_k})$ converges uniformly to a convex function $v$ on $\tir{\Omega}$.
\end{theorem}

\begin{proof}

Recall from Theorem \ref{geo-proof} that $\partial \Gamma(u_h) ((\Conv (N_h))^\circ) \subset \chi_{\Gamma(u_h)}(\mathcal{N}_h)$. Since $\chi_{\Gamma(u_h)}(\mathcal{N}_h)$ is uniformly bounded, there is a constant $C$ independent of $h$ such that for all $p \in \partial \Gamma(u_h) ((\Conv (N_h))^\circ)$, $||p||\leq C$. As $|u_h| \leq M$ for a constant $M$ independent of $h$, $\Gamma(u_h)$ is uniformly bounded. From the definition \eqref{extension} of convex extension, $\widetilde{\Gamma}(u_h)$ is uniformly bounded on $\widetilde{\Omega}$. As a convex function, $\widetilde{\Gamma}(u_h)$ is uniformly Lipschitz on $\tir{\Omega}$. The result follows from the Arzela-Ascoli theorem as in the proof of Theorem \ref{mass-compact03}.

\Qed
\end{proof}

We now discuss uniform convergence properties of  a subsequence of $u_h$. We make the additional assumption that $u_h$ is discrete convex. We first construct an interpolant of $u_{h_k}$ defined on $\tir{\Omega}$. 

We recall that a triangulation is conforming if it is a partition of $\tir{\widetilde{\Omega}}$ into convex hulls of $d+1$ points, and the intersection of two elements is the convex hull of $0 \leq k < d+1$ points.

\begin{lemma} \label{triangulation-lem}
For a mesh function $u_h$, one can find a conforming triangulation $\widetilde{\mathcal{T}}_h$ of $\widetilde{\Omega}$ such that each element of $\mathcal{N}_h$ is a vertex of an element of $\widetilde{\mathcal{T}}_h$ and $\Conv(\mathcal{N}_h) $ is the union of elements of $\widetilde{\mathcal{T}}_h$ and any boundary vertex of $\Conv(\mathcal{N}_h) $ is an element of $\mathcal{N}_h$.  Moreover $\widetilde{\Gamma}(u_h)$ is piecewise linear on $\widetilde{\mathcal{T}}_h$.  

\end{lemma}

\begin{proof} Let $\widetilde{\mathcal{T}}(\Gamma(u_h))$ denote a triangulation induced by $\widetilde{\Gamma}(u_h)$ on $\widetilde{\Omega}$. Recall that $\widetilde{\Gamma}(u_h)$ is piecewise linear on $\widetilde{\mathcal{T}}(\Gamma(u_h))$. Recall also that $\mathcal{T}(\Gamma(u_h))$ denote a triangulation induced by  $\Gamma(u_h)$ onto $\Conv(\mathcal{N}_h) $. 
We use the points of $\mathcal{N}_h$ which are not vertices of $\mathcal{T}(\Gamma(u_h))$ to create a conforming triangulation of $\Conv(\mathcal{N}_h) $ by subdividing elements of  $\mathcal{T}(\Gamma(u_h))$. This gives a triangulation $\mathcal{T}_h$ of $\Conv(\mathcal{N}_h) $ with set of vertices equal to $\mathcal{N}_h$ and $\Gamma(u_h)$ is piecewise linear on $\mathcal{T}_h$. Finally elements of $\widetilde{\mathcal{T}}(\Gamma(u_h))$ are subdivided to form a triangulation of $\widetilde{\Omega} \setminus \Conv(\mathcal{N}_h)$ with vertices in the union of $\tir{\widetilde{\Omega}} \cap \mathbb{Z}^d_h$ and the boundary vertices of $\mathcal{T}_h$. This results in the desired triangulation $\widetilde{\mathcal{T}}_h$ of $\widetilde{\Omega}$. 

\Qed

\end{proof}

Let $I(u_h)$ be the piecewise linear continuous function which is equal to $u_h$ on $\mathcal{N}_h$ and equal to $\widetilde{\Gamma}(u_h)$ at the other vertices of $\widetilde{\mathcal{T}}_h$. By construction $I(u_h) = \Gamma(u_h)$ on the boundary of $\Conv(\mathcal{N}_h) $ when $u_h=\Gamma(u_h)$ on $\partial \Omega_h$. We note that $I(u_h)$ may not be convex.


\begin{lemma} \label{univariate}

Let $u_{h}$ be discrete convex. The univariate piecewise linear interpolant $U$ of $u_h$ on the line $L$ through $x_0  \in \Omega_h$ and direction $e \in V$ is convex. 
\end{lemma}

\begin{proof} 
We define for $x \in L \cap \Omega$
$$
\partial U(x)  =\{ \, q \in \R, U(y) - U(x) \geq q e \cdot (y-x) \text{ for all } y \in L \cap \Omega \, \}.
$$

For $x \in L\cap \Omega_h$, put $x_+=x+h_x^e e$ and $x_- = x-h_x^{-e} e$. We claim that 
\begin{align*}
\partial U(x) & = \frac{1}{||e||^2}\bigg[  \frac{u_h(x)-u_h(x_- ) }{h^{-e}_x} , \frac{u_h(x_+ ) - u_h(x)}{h^e_x}\bigg],  x \in  L\cap \Omega_h \\
\partial U(y) & = \frac{1}{h_x^e ||e||^2 }(u_h(x_+) - u_h(x)) \ \text{ for } y \in (x, x_+), x \in  L\cap \Omega_h\\
\partial U(y) & = \frac{1}{h_x^{-e} ||e||^2 }(u_h(x) - u_h(x_-)) \ \text{ for } y \in (x_-, x), x \in  L\cap \Omega_h.
\end{align*}
Since for $x \in L\cap \Omega_h$, $U$ is linear on $[x,x_+]$ and $[x_-, x]$, for $y \in (x, x_+)$ or $y \in (x_-, x)$, $\partial U(y)$ is reduced to at most one point. 
Let $x \in L\cap \Omega_h$. We show that $p=(u_h(x_+) - u_h(x))/(h_x^e ||e||^2) \in \partial U(x)$. The other cases are similar. 

Let $k, l \geq 0$ be integers such that $x+kh e \in \Omega_h$, $x+(k+1)h e \not \in \Omega_h$, $x-lh e \in \Omega_h$ and $x-(l+1)h e \not \in \Omega_h$. We have by discrete convexity
\begin{multline*}
\frac{u_h(x+kh e + h^e_{x+kh e} e ) - u_h(x+kh e) }{h^e_{x+kh e}} \geq \frac{ u_h(x+kh e) -u_h(x+(k-1)h e)}{h} \geq \\ \ldots \geq 
\frac{u_h(x+h) - u_h(x) }{h} \geq \frac{u_h(x) - u_h(x-h) }{h}  \geq \ldots \geq \\ \frac{u_h(x-(l-1)h e) - u_h(x- l h e)}{h} 
\geq \frac{u_h(x- l h e) - u_h(x- l h e -h^{-e}_{x-lhe} e) }{ h^{-e}_{x-lhe} }.
\end{multline*}
Using the above equation and an induction argument we have for $0 \leq r \leq k$, $u_h(x+rh e) - u_h(x) \geq k (u_h(x+h)-u_h(x))$ and for $0 \leq s \leq l$, $u_h(x-sh) - u_h(x) \geq (-s) (u_h(x+h)-u_h(x))$.  This gives for all $z \in L \cap \Omega_h$, $u_h(z)-u_h(x) \geq p e \cdot(z-x)$. With a similar argument, this also holds for $z=x+kh e + h^e_{x+kh e} e$ and $z=x- l h e -h^{-e}_{x-lhe} e$. We conclude that
$p=(u_h(x_+) - u_h(x))/(h_x^e ||e||^2) \in \partial U(x)$.

Since  $\partial U(x) \neq \emptyset$ for all $x \in L \cap \Omega$, $U$ is convex. To see this, let $x_1, x_2 \in L \cap \tir{\Omega}$ and $\lambda \in (0,1)$. Put $x=\lambda x_1 + (1-\lambda) x_2$ and choose $p \in \partial U(x)$. We have
\begin{align*}
U(x_1) & \geq U(x)+ p e \cdot (x_1-x) \\
U(x_2) & \geq U(x)+ p e \cdot (x_2-x). 
\end{align*}
Therefore $\lambda U(x_1) + (1- \lambda) U(x_2) \geq U(x)$. Since $U$ is continuous, this also holds for $\lambda \in [0,1]$.

\Qed

\end{proof}

\begin{theorem} \label{loc-u}

 Let $u_h$ be a family of discrete convex functions  such that $|u_h| \leq M$ for a constant $M$ independent of $h$. The interpolant $I(u_h)$ is locally uniformly Lipschitz on $\Omega$. As a consequence there is a subsequence $h_k$ such that both
 $\Gamma(u_{h_k})$ and $u_{h_k}$ converge uniformly on compact subsets of $\Omega$ to continuous convex functions on $\Omega$.

\end{theorem}

\begin{proof}

Let $K$ be a compact subset of $\Omega$ and choose $h_K$ such that $K \subset \Conv(\Omega_{h_K})$. Let $\delta_K >0$ such that the $\delta_K$ neighborhood of $\Conv(\Omega_{h_K})$ is contained in $\Omega$. Given $x_0 \in \Omega_{h_K}$, by Lemma \ref{univariate}, we have $| u_h(x) - u_h(y)| \leq 2M/\delta_K ||x-y||$ for all $x, y \in \Omega_{h_K}$, $x=x_0 + \alpha h e$, 
$y=x_0 + \beta h e$ for some integers $\alpha$ and $\beta$ and $e$ an element of the canonical basis of $\R^d$.

Let us now assume that $y=x+\sum_{i=1}^d m_i h r_i$ for integers $m_i, i=1,\ldots,d$. We have
\begin{align*}
u_h(x)-u_h(y) = \sum_{k=0}^{d-1} u_h\bigg(x+\sum_{i=1}^k m_i h r_i \bigg) -  u_h\bigg(x+\sum_{i=1}^{k+1} m_i h r_i \bigg),
\end{align*}
where $\sum_{i=1}^0 m_i h r_i =0$. We obtain $|u_h(x)-u_h(y)| \leq 2M h/\delta_K \sum_{k=0}^{d-1} |m_{k+1}| = 2M/\delta_K ||x-y||_1 \leq 2M d/\delta_K ||x-y||$ where $||.||_1$ denotes the 1-norm. We conclude that $u_h$ is uniformly Lipschitz on $\Conv(\Omega_{h_K})$. 
We claim that the interpolant $I(u_h)$ is thus locally uniformly Lipschitz on $\Omega$. 

As in the proof of Theorem \ref{mass-compact03}, we obtain a subsequence $h_k$ such that $I(u_{h_k})$ and hence $u_{h_k}$ converges uniformly on compact subsets a function $w \in C(\Omega)$. 
We prove in Lemma \ref{unif-disc-conv} below that $w$ is convex. Applying again Theorem \ref{mass-compact03}, we have a subsequence also denoted $h_k$
such that  $\Gamma(u_{h_k})$ converges uniformly on compact subsets of $\Omega$ to a continuous convex function on $\Omega$.

To see that $I(u_h)$ is uniformly Lipschitz on $\Conv(\Omega_{h_K})$, we first note that $I(u_h)$ is uniformly Lipschitz on each element $T$ with vertices in $\Conv(\Omega_{h_K})$.  Let $\<v_1,\ldots,v_{d+1}\>$ denote the vertices of $T$. There exists $\zeta_i \in T$ such that for all $i=2,\ldots,d+1$, 
$I(u_h)(v_i)- I(u_h)(v_1) = D I(u_h) (\zeta_i) \cdot (v_i-v_1)$. Since $I(u_h)$ is linear on $T$, $D I(u_h) (\zeta_i)=D I(u_h) (\zeta_2)\equiv p$ for all $i$. Now, 
$$
| I(u_h)(v_i)- I(u_h)(v_1) | = |u_h(v_i) - u_h(v_1)| \leq C ||v_i-v_1||, i=2,\ldots,d+1.
$$
Therefore $|p\cdot (v_i-v_1)| \leq C ||v_i-v_1||, i=2,\ldots,d+1$. Since $T$ is non degenerate, we obtain
$||p|| \leq C$ with $C$ independent of $h$. Again, by the linearity of $I(u_h)$ on $T$, we can find a vector $b \in \R^d$ such that for all $x \in T$, $I(u_h)(x) = p \cdot x+b$. This implies that for all $x,y \in T$, $|I(u_h)(x)  - I(u_h)(y) | \leq C ||y-x||$. 

Next, let $x, y \in \Conv(\Omega_{h_K})$. Put $x_1=x$ and $x_{N+1}=y$. The line through $x$ and $y$ intersects a sequence of elements $T_1, \ldots, T_N$ of $\widetilde{\mathcal{T}}_h$ with $x_1 \in T_1$ and $x_{N+1} \in T_N$. 
We may assume that $T_i \subset \Conv(\Omega_{h_K})$ for $i=1,\ldots,N$ by taking a further subdivision of the triangulation obtained from  Lemma \ref{triangulation-lem}.
Let $x_i, i=2,\ldots,N$ be points on the line such that $x_i \in \partial T_{i-1} \cap \partial T_{i}$.  As $I(u_h)$ is linear on  $T_i$ for all $i$, we have $| I(u_h)(x_i) - I(u_h)(x_{i-1})| \leq C ||x_i - x_{i-1}||$ for all $i$. Since the points $x_i$ are colinear, 
$||y-x|| = \sum_{i=2}^{N+1} ||x_i-x_{i-1}||$, from which the uniform Lipschitz property on $\Conv(\Omega_{h_K})$ follows.

\Qed

\end{proof}

A priori, the subsequences from Theorem \ref{loc-u} may converge to different limits as there could be points $x$ where $u_h(x) > \Gamma(u_h)(x)$. We will give below conditions under which the limits are the same. 

\begin{theorem} \label{mass-compact04}
Assume that
$|u_h| \leq M$ for a constant $M$ independent of $h$ and $\chi_{\Gamma(u_h)}(\mathcal{N}_h)$ is uniformly bounded. Assume furthermore that $u_h$ is uniformly Lipschitz on $\tir{\Omega}$.  Then $I(u_h)$ is uniformly bounded on $\tir{\Omega}$ and uniformly Lipschitz on $\tir{\Omega}$. As a consequence, there is a subsequence $h_k$ such that both
 $\Gamma(u_{h_k})$ and $u_{h_k}$ converge uniformly to a convex function $v$ on $\tir{\Omega}$.
\end{theorem}

\begin{proof}

As in the proof of Theorem \ref{loc-u}, the interpolant $I(u_h)$ is uniformly Lipschitz on $\Conv(\mathcal{N}_h)$. Next, we note that $\widetilde{\Gamma}(u_h)$ is uniformly Lipschitz on  $\widetilde{\Omega}$. Indeed, by Theorem \ref{geo-proof}, $\chi_{\Gamma(u_h)}(\R^d) \subset  \Conv \big( \chi_{\Gamma(u_h)}(\mathcal{N}_h)\big)$ and is therefore uniformly bounded. By 
\cite[Lemma 1.1.6]{Guti'errez2001} $\widetilde{\Gamma}(u_h)$ is uniformly Lipschitz on $\widetilde{\Omega}$. We reproduce here the proof. 

Since $\chi_{\Gamma(u_h)}(\R^d)$ is bounded, 
for all $x \in \widetilde{\Omega}$ and $p \in \chi_{\Gamma(u_h)}(x)$, $||p|| \leq C$ with $C$ independent of $h$.
For $x, y \in \widetilde{\Omega}$ and $p \in  \chi_{\Gamma(u_h)}(x)$, $\widetilde{\Gamma}(u_h)(y) -  \widetilde{\Gamma}(u_h)(x) \geq p \cdot (y-x) \geq - ||p|| \, ||y-x|| \geq -C \, ||y-x||$. Reversing the roles of $x$ and $y$, we obtain $\widetilde{\Gamma}(u_h)(y) -  \widetilde{\Gamma}(u_h)(x)| \leq C ||y-x||$ for all $x, y \in \widetilde{\Omega}$.

By construction, $I(u_h) = \Gamma(u_h)$ on $\Omega_1= \cup \{ \, T \in \widetilde{\mathcal{T}}_h,  T \not\subset \Conv(\mathcal{N}_h) \, \}$. Since $\widetilde{\Gamma}(u_h)$ is uniformly Lipschitz on  $\widetilde{\Omega}$, for all $x, y \in \Omega_1$, we have $|I(u_h)(x)  - I(u_h)(y) | \leq C ||y-x||$. It remains to consider the case $x \in \Omega_1$ and $y \in \Conv(\mathcal{N}_h)$. We consider the line through $x$ and $y$. It intersects $\partial \Conv(\mathcal{N}_h)$ at a point $z$. 

We have by construction $I(u_h)(z) = \Gamma(u_h)(z)$ because the triangulation of Lemma \ref{triangulation-lem} is a subdivision of a triangulation induced by 
$\widetilde{\Gamma}(u_h)$ on $\widetilde{\Omega}$. Thus
$|I(u_h) (x) - I(u_h) (y)| \leq |I(u_h) (x) - I(u_h) (z)| + |I(u_h) (z) - I(u_h) (y)| \leq ||x-z||+||z-y|| = ||x-y||$ where we used the uniform Lipschitz property on $\Omega_1$ and $\Conv(\mathcal{N}_h)$ as well as the colinearity of $x, y$ and $z$. 

We conclude that $I(u_h)$ is uniformly Lipschitz on $\tir{\Omega}$, and uniformly bounded on $\tir{\Omega}$ since $|u_h|\leq M$. 

Using Theorem \ref{mass-compact02} we obtain a subsequence $u_{h_k}$ such that $\Gamma(u_{h_k})$ converges uniformly to a convex function $v$ on $\tir{\Omega}$. By the Arzela-Ascoli theorem, we have a further subsequence also denoted $u_{h_k}$ such that $u_{h_k}$ converges uniformly to a function $w$ on $\tir{\Omega}$. As a uniform limit of the continuous functions $I(u_{h_k})$, $w \in C(\tir{\Omega})$.
We show below in Lemma \ref{unif-disc-conv}, that $w$ is convex on $\tir{\Omega}$.
By Theorem \ref{cvg-cvx}, $\Gamma(u_{h_k})$ converges to $w$ uniformly on compact subsets of $\Omega$.  Thus $v=w$ on $\Omega$ and since $v$ and $w$ are continuous as uniform limit of continuous functions $\Gamma(u_{h_k})$ and $I(u_{h_k})$, $v=w$ on $\tir{\Omega}$.

\Qed

\end{proof}

The Lipschitz continuity of $u_h$ on $\tir{\Omega}$ holds for certain discretizations of the second boundary value problem for the Monge-Amp\`ere equation. 
For the Dirichlet problem, the following lemma gives conditions under which $u_h$ is Lipschitz continuous on $\tir{\Omega}$. The condition that $\chi_{\Gamma(u_h)}(\mathcal{N}_h)$ is uniformly bounded may not be necessary since Dirichlet boundary values are prescribed. With Lemma \ref{Lip-lem} below, one obtains a subsequence $u_{h_k}$ which converges uniformly on $\tir{\Omega}$.

\begin{lemma} \label{Lip-lem} Let $u_h$ be a family of discrete convex functions such that $\chi_{\Gamma(u_h)}(\mathcal{N}_h)$ is uniformly bounded. Assume that $u_h=\Gamma(u_h)$ on $\partial \Omega_h$.
There exists a constant $C$ independent of $h$ such that for all $x, y \in \mathcal{N}_h$, we have
$$
|u_h(x) - u_h(y)| \leq C ||y-x||,
$$ 
with $C$ independent of $h$.
\end{lemma}

\begin{proof} A related proof can be found in \cite[Proposition 4.3]{benamou2017minimal}. Recall first from Theorem \ref{mass-compact02} that 
$\widetilde{\Gamma}(u_h)$ is Lipschitz continuous on $\tir{\Omega}$
We first prove that the Lipschitz continuity property holds for $x, y \in \Omega_h$.

Let $x \in \Omega_h$ and $e \in V$ such that $x \pm h e \in \Omega_h$. Since $\Delta_{e} v_h(x) \geq 0$, we have
$$
v_h(x+h e) - v_h(x) \geq v_h(x) - v_h(x-h e). 
$$
Therefore for integers $k$ and $l$ such that $k \geq l$, $x+k h e$ and $x+l he$  are in $\Omega_h$
\begin{equation} \label{l1}
v_h(x+k he) - v_h(x+ (k-1) h e) \geq v_h(x+ (l+1) he) - v_h(x+ l h e). 
\end{equation}
Let $k$ be the maximum integer such that $x+(k+1) h e \notin \Omega_h$ and assume that $l$ is the smallest integer such that $x+(l-1) h e \notin \Omega_h$. Then both $x+ k h e + h^e_{x+ k h e} e$ and $x+ l h e - h^{-e}_{x+ l h e} e$ are on $\partial \Omega_h$. By the Lipschitz continuity of $\widetilde{\Gamma}(u_h)$ on $\tir{\Omega}$, the assumption that $u_h=\Gamma(u_h)$ on $\partial \Omega_h$,  \eqref{g-u}, $\Delta_e u_h(x+ k h e) \geq 0$, \eqref{l1} and $\Delta_e u_h(x+ l h e) \geq 0$, we have
\begin{multline} \label{p4}
C ||e|| \geq \frac{\Gamma(u_h)( x+ k h e + h^e_{x+ k h e} e) - \Gamma(u_h)(x+ k h e)}{h^e_{x+ k h e}}  = \\
\frac{u_h( x+ k h e + h^e_{x+ k h e} e) - \Gamma(u_h)(x+ k h e)}{h^e_{x+ k h e}} \geq \\ \frac{u_h( x+ k h e + h^e_{x+ k h e} e) - u_h(x+ k h e)}{h^e_{x+ k h e}} \geq 
\frac{u_h(x+k he) - u_h(x+ (k-1) h e)}{h}  \geq \\ \frac{u_h(x+ (l+1) he) - u_h(x+ l h e)}{h} \geq 
\frac{u_h(x+ l h e) - \Gamma(u_h)( x+ l h e - h^{-e}_{x+ l h e} e) }{h^{-e}_{x+ l h e}} \geq \\
 \frac{\Gamma(u_h)(x+ l h e) - \Gamma(u_h)( x+ l h e - h^{-e}_{x+ l h e} e) }{h^{-e}_{x+ l h e}} \geq - C ||e||.
\end{multline}
We conclude that for an integer $m$ such that $x+m h e \in \Omega_h$, we have $|u_h(x+m h e ) - u_h(x)| \leq C |m| h ||e||$. 
Put $y=x+\sum_{i=1}^d m_i h r_i$ where we recall that $(r_1,\ldots,r_d)$ denotes the canonical basis of $\R^d$ and each $r_i$ is in $V$ by assumption. We obtain for $x, y \in \Omega_h$,
\begin{equation*} 
| u_h(y) - u_h(x)| \leq C h  \sum_{i=1}^d |m_i| = C ||y-x||_1 \leq C ||y-x||,
\end{equation*}
where $||.||_1$ denotes the 1-norm. 

Next, again by the assumption that $u_h=\Gamma(u_h)$ on $\partial \Omega_h$ and the Lipschitz continuity $\widetilde{\Gamma}(u_h)$ on $\tir{\Omega}$, for $x, y \in \partial \Omega_h$
\begin{equation*} 
|u_h(y) - u_h(x)| = |\Gamma(u_h)(y) - \Gamma(u_h)(x)| \leq C ||y-x||.
\end{equation*}
Finally, assume that $x \in \Omega_h$ and $y \in \partial \Omega_h$. Let $z \in \Omega_h$ such that $y=z+h^e_z e$ for some $e \in V$. Arguing as for \eqref{p4}, we obtain $|u_h(y) - u_h(z)| \leq C h^e_z ||e||$. Put $e=e_1$ and let $\{e_1,\ldots,e_d \, \}$ be a basis of $\mathbb{Z}^d$. We can write $z=x+\sum_{i=1}^d m_i h e_i$ and so
$y= m_1 h e_1 + h^e_z e_1 + \sum_{i=2}^d m_i h e_i$. Arguing as above, we obtain $|u_h(y) - u_h(x)| \leq C ||y-x||$ in this case as well. This completes the proof. \Qed

\end{proof}



Next, we state a result which is implicit in convergence studies of the Monge-Amp\`ere equation. As indicated in the preliminaries, the set of directions is now taken to be $V=\mathbb{Z}^d$ and we focus on the limit $h \to 0$. 

\begin{lemma} \label{unif-disc-conv}
Take $V=\mathbb{Z}^d\setminus \{\, 0 \, \}$ and let $u_{h}$ be discrete convex. If $u_h$ converges uniformly on compact subsets of $\Omega$ to a function $u \in C(\Omega)$, $u$ is convex on $\Omega$. 
\end{lemma}

\begin{proof} 

Let $e \in \mathbb{Z}^d$ and $x_0 \in \Omega_h$. We recall from Lemma \ref{univariate} that the univariate piecewise linear interpolant $U$ of $u_h$ on the line $L$ through $x_0$ and direction $e$ is convex. Fix $x, y \in \Omega$ and let for $\theta \in [0,1]$, $u_{x,y}(\theta) = u(\theta y + (1-\theta) x)$. The convexity of $u_{x,y}$ would imply that 
$$
u(\theta y + (1-\theta) x) = u_{x,y}(\theta) \leq \theta u_{x,y}(1) + (1-\theta) u_{x,y}(0) =  \theta u(y) + (1-\theta) u(x).
$$
It is therefore enough to show that $u_{x,y}$ is convex. Put $\zeta=y-x$. We have $u_{x,y}(\theta) =u(x+\theta\zeta )$. Let $x_h \in \Omega_h$ such that $x_h \to x$ and choose $\zeta_h \in \mathbb{Z}^d_h$ such that $\zeta_h \to \zeta$. Denote by $L_h$ the line through $x_h$ and direction $\zeta_h$.  

Given $\alpha, \beta, \theta \in (0,1)$, let $a_h$ and $b_h$ be points on $L_h \cap \mathcal{N}_h$ of minimum distance to $x_h+\alpha \zeta_h$ and $x_h + \beta \zeta_h$ respectively. As $h\to 0$, $a_h \to x+\alpha \zeta$ and $b_h \to x+\beta \zeta$. Now, let 
$c_h \in L_h \cap \mathcal{N}_h$ of minimum distance to $x_h+ \theta \alpha \zeta_h+ (1-\theta) \beta \zeta_h$. As $h \to 0$,
$c_h \to x+ \theta \alpha \zeta+ (1-\theta) \beta \zeta$. 

Recall that the piecewise linear interpolant of $u_h$ on the line $L_h$ is convex. That is, for $\theta_h \in [0,1]$ such that $c_h=\theta_h a_h +  (1-\theta_h) b_h$ we have
$$
u_h(c_h) \leq \theta_h u_h(a_h) +  (1-\theta_h) u_h(b_h). 
$$
Up to a subsequence $\theta_h \to \theta' \in [0,1]$. By the uniform convergence of $u_h$ to $u$ and the continuity of $u$, we obtain
$$
u(x+ \theta \alpha \zeta+ (1-\theta) \beta \zeta) \leq \theta' u(x+ \alpha \zeta) +  (1-\theta') u(x+ \beta \zeta).
$$
Since $c_h=\theta_h a_h +  (1-\theta_h) b_h$ we get $x+ \theta \alpha \zeta+ (1-\theta) \beta \zeta) = \theta'(x+\alpha\zeta ) + (1-\theta')(x+\beta\zeta )$ which gives $\theta=\theta'$ when $\alpha \neq \beta$ and $\zeta\neq 0$. The cases $\alpha=\beta$ and $\zeta = 0$ are trivial.

We have shown that $u_{xy}$ is convex. This implies that $u$ is convex on $\Omega$. 

\Qed
\end{proof}

We now prove that if in addition to the assumptions of Theorem \ref{loc-u}, $u_h = g$ on $\partial \Omega_h$ for a convex function $g \in C(\tir{\Omega})$ and $\omega(1,u_h,\Omega_h)$ is uniformly bounded, there are subsequences $\Gamma(u_{h_k})$ and $u_{h_k}$ converging uniformly on compact subsets of $\Omega$ to the same convex function $v \in C(\Omega)$. Moreover $v=g$ on $\partial \Omega$. 
Recall the discrete Laplacian
$$
\Delta_h v_h(x) =  \sum_{i=1}^d  \Delta_{r_i} v_h(x).
$$

For $x \in \tir{\Omega}$ define
\begin{equation} \label{bd-data}
U(x) = \sup \{ \, L(x), L \leq g \text{ on } \partial \Omega, L \text{ affine} \, \},
\end{equation}
the convex envelope with boundary data $g$. We have by \cite[Theorem 5.2]{Hartenstine2006}

\begin{theorem} \label{L-bd}
For $g \in C(\tir{\Omega})$, the function $U$ defined by
\eqref{bd-data} is in $C(\tir{\Omega})$ and $U=g$ on $ \partial \Omega$.
\end{theorem}

\begin{theorem} \label{main2}
Let $u_h$ be a family of discrete convex functions which is uniformly bounded and with Monge-Amp\`ere masses uniformly bounded. Assume that $u_h = g$ on $\partial \Omega_h$ for a convex function $g \in C(\tir{\Omega})$. There is a subsequence $h_k$ such that $\Gamma(u_{h_k})$ and $u_{h_k}$ converge uniformly on compact subsets of $\Omega$ to a convex function $v \in C(\Omega)$. Moreover $v=g$ on $\partial \Omega$. 

\end{theorem}

\begin{proof}
We recall that by Lemma \ref{cv-bd0}, since $u_h = g$ on $\partial \Omega_h$ for a convex function $g \in C(\tir{\Omega})$ which can then be extended to $\R^d$ as a convex function, we have $\Gamma(u_h)=u_h=g$ on $\partial \Omega_h$. 
By Theorem \ref{loc-u} there is a subsequence $h_k$ such that $\Gamma(u_{h_k})$ and $u_{h_k}$ converge uniformly on compact subsets of $\Omega$ to continuous convex functions $u$ and $v$ respectively. We show that 
$u=v=g$ on $\partial \Omega$.

We first prove that for $\zeta \in \partial \Omega$, $\lim_{x \to \zeta} v(x) \geq g(\zeta)$ by arguing as in the proof of \cite[Lemma 5.1]{Hartenstine2006}. Let $\epsilon >0$. By Theorem \ref{L-bd} there exists an affine function $L$ such that $L \leq g$ on $\partial \Omega$ and $L(\zeta) \geq g(\zeta) - \epsilon$. Put $z_h=u_h-L$. Since $u_h=g$ on $\partial \Omega_h$, we have $z_h \geq 0$ on $\partial \Omega_h$. 

Now let $x \in \Omega$ and $x_h \in \Omega_h$ such that $x_h \to x$. Since $u_h$ converges to $v$ uniformly on compact subsets of $\Omega$, $z_h$ converges to $z$ uniformly on compact subsets of $\Omega$ and thus $z_h(x_h) \to v(x)-L(x) \coloneqq z(x)$. Assume that $z(x) <0$. 
By the discrete Aleksandrov's maximum principle Lemma \ref{d-stab}
 applied to $z_h$ we have
\begin{align*}
\begin{split} 
(-z_h(x_h))^d 
 & \leq C d(x_h, \partial \Omega) (\diam(\Omega))^{d-1}\omega(1,u_h,\Omega_h) \\
 & \leq C d(x_h, \partial \Omega) \omega(1,u_h,\Omega_h) \leq C ||x_h-\zeta|| \omega(1,u_h,\Omega_h).
 \end{split}
\end{align*}
By the assumption on the Monge-Amp\`ere masses, $\omega(1,u_h,\Omega_h) \leq C$ with $C$ independent of $h$. Then
\begin{equation} \label{alex-applied}
(-z_h(x_h))^d \leq C ||x_h-\zeta||.
\end{equation}
Taking the limit as $h_k \to 0$ in \eqref{alex-applied}, we obtain for each $x \in \Omega$ for which $z(x) <0$
\begin{align*}
(-z(x))^d &\leq C ||x-\zeta||.
\end{align*}
In summary
$$
\text{either} \ z(x) \geq - C ||x-\zeta||^{\frac{1}{d}} \, \text{or} \, z(x) \geq 0, x \in \Omega.
$$
We conclude that
$$
v(x) \geq L(x) -C ||x-\zeta||^{\frac{1}{d}} \, \text{on} \, \Omega.
$$
Taking the limit as $x \to \zeta$ we obtain $\lim_{x \to \zeta} v(x) \geq L(\zeta) \geq g(\zeta) - \epsilon$. Since $\epsilon$ is arbitrary, we have $\lim_{x \to \zeta} v(x) \geq g(\zeta)$.


Next, we prove that $\lim_{x \to \zeta} v(x) \leq g(\zeta)$. Let $w_h$ denote the solution of the problem 
\begin{equation}\label{d01h}
\Delta_h w_h=0 \text{ on } \Omega_h \text{ with } w_h=g \text{  on } \partial \Omega_h.
\end{equation}
Since the above problem is linear, for existence and uniqueness of $w_h$, it is enough to prove uniqueness, i.e. the problem 
$\Delta_h s_h=0$ on $\Omega_h$ with $s_h=0$ on $\partial \Omega_h$ has the unique solution $s_h=0$. This follows from the discrete maximum principle for the discrete Laplacian \cite[Theorem 4.77]{Hackbusch2010}. Indeed, as $\Delta_h s_h \geq 0$ on $\Omega_h$ with $s_h=0$ on $\partial \Omega_h$ we obtain $s_h \leq 0$ on $\Omega_h$. Using $\Delta_h s_h \leq 0$, we obtain $s_h \geq 0$ on $\Omega_h$. This shows that $s_h=0$, proving existence and uniqueness of $w_h$.

Since $u_h$ is discrete convex, we have $\Delta_h u_h \geq 0$. Therefore $\Delta_h (u_h -w_h) \geq 0$ on $\Omega_h$ with $u_h-w_h=0$ on $\partial \Omega_h$. Again, by the discrete maximum principle for the discrete Laplacian \cite[Theorem 4.77]{Hackbusch2010}, we have
$u_h-w^{}_h \leq 0$ on $\Omega_h$. 

Since a convex domain is Lipschitz, we can apply the results of \cite[section 6.2 ]{del2018convergence} and claim that 
 $w_h$ converges uniformly on compact subsets to the unique viscosity solution of $\Delta w =0$ on $\Omega$ with $w=g$ on
$\partial \Omega$, c.f. section \ref{appendix} for additional details. This gives $v(x) \leq w(x)$ on $\Omega$. It is also proved in \cite{del2018convergence} that $w \in C(\tir{\Omega})$. We conclude that  $\lim_{x \to \zeta} v(x) \leq g(\zeta)$. Thus
$v \in C(\tir{\Omega})$ and $v=g$ on $\partial \Omega$.

Now, from Theorem \ref{equi-measures}, we have $\omega(1,\Gamma(u_h),\Omega_h) =\omega(1,u_h,\Omega_h) \leq C$. We can then apply the same arguments to $\Gamma(u_h)$ as a mesh function, to obtain $u \in C(\tir{\Omega})$ and $u=g$ on $\partial \Omega$. Again, by Theorem \ref{equi-measures} we have $\omega(1,u,.) = \omega(1,v,.)$. By unicity of the solution of the Dirichlet problem for the Monge-Amp\`ere equation \cite[Theorem 1.1]{Hartenstine2006}, we have $u=v$ on $\tir{\Omega}$.

 \Qed

\end{proof}


\section{Appendix} \label{appendix}

In \cite{del2018convergence} convergence of numerical schemes for discrete approximations  to viscosity solutions of the Dirichlet problem for dynamic programming principles for the $p$-Laplacian is given. The proof is based on a numerical analysis approach,  the Barles-Souganidis framework which consists in checking stability, consistency and monotonicity of the numerical scheme. It is known that the Barles-Souganidis framework requires the so-called strong uniqueness property for the differential equation, which is a comparison principle for equations with the boundary condition in the viscosity sense. In \cite{del2018convergence} the strong uniqueness property for the Laplace equation is proved for smooth domains and convergence of the numerical scheme on a bounded Lipschitz domain is obtained through barriers on appropriate shrinking rings. We review below their approach for the standard 
discretization of the Laplace equation.


Recall that $\Omega \subset \R^d$ is a bounded Lipschitz domain 
and $g \in C(\partial \Omega)$. We consider the problem
\begin{align} \label{d1}
\begin{split}
-\Delta u & = 0 \ \text{ in } \Omega \\
u & = g \ \text{ on } \partial \Omega.
\end{split}
\end{align}
We define for $h>0$ the following sets:

Outer boundary strip: $\Gamma_{h} = \{ \, x \in \R^d\setminus \Omega, d(x,\partial \Omega) \geq h\,  \}$ and put $O=\Gamma_1$.

Inner boundary strips:  $I_{h} = \{ \,  x \in  \Omega, d(x,\partial \Omega) \leq h \,  \}$ and put $I=I_1$.

Extended domain: $\widetilde{\Omega} = \Omega \cup O$ and extended computational domain $\widetilde{\Omega}_h = \Omega_h \cup O$.

Let $G$ be a continuous extension of $g$ to $\widetilde{\Omega}$. Note that $\partial \Omega_h \subset O$. Given a mesh function $u_h$, we extend it to $\Omega_h \cup O$ by $u_h(x)=G(x)$ for all $x \in O \setminus \partial \Omega_h$.  Analogous to \cite[(2.9)]{del2018convergence}, we consider the discrete problem: find a mesh function $u_h$ such that
\begin{align}
\begin{split} \label{d11h}
-\Delta_h u_h & = 0 \ \text{ on } \Omega_h \\
u_h & = G  \ \text{ on } O.
\end{split}
\end{align} 
A mesh function solves \eqref{d01h} if and only if it solves \eqref{d11h}.

In \cite{del2018convergence} a notion of viscosity solution of \eqref{d1} is first given. 
The authors therein recall the existence and uniqueness of such a viscosity solution. They then introduce a generalized version of viscosity solution where the boundary condition is assumed in a weaker sense. That notion has become known as boundary condition in the viscosity sense.

An upper semi-continuous function on $\tir{\Omega}$ is a viscosity subsolution of \eqref{d1} if whenever $x_0 \in \tir{\Omega}$ and $\phi \in C^2(\tir{\Omega})$ satisfy
$$
\phi(x_0) = u(x_0), (u-\phi)(x)< (u-\phi)(x_0) \text{ for } x\neq x_0, 
$$
we have
\begin{align}
-\Delta \phi(x_0) & \leq 0 \ \text{ if } x_0 \in \Omega\\
u(x_0) - g(x_0) & \leq 0 \ \text{ if } x_0 \in \partial \Omega. \label{d1b}
\end{align}
If the condition \eqref{d1b} at the boundary is replaced by
\begin{equation}
\min \{ \, -\Delta \phi(x_0), u(x_0) - g(x_0) \, \} \leq 0, x_0 \in \partial \Omega,
\end{equation}
the function $u$ is said to be a generalized viscosity subsolution of \eqref{d1}.

A lower semi-continuous function on $\tir{\Omega}$ is a viscosity supersolution of \eqref{d1} if whenever $x_0 \in \tir{\Omega}$ and $\phi \in C^2(\tir{\Omega})$ satisfy
$$
\phi(x_0) = u(x_0), (u-\phi)(x)> (u-\phi)(x_0) \text{ for } x\neq x_0, 
$$
we have
\begin{align}
-\Delta \phi(x_0) & \geq 0 \ \text{ if } x_0 \in \Omega\\
u(x_0) - g(x_0) & \geq 0 \ \text{ if } x_0 \in \partial \Omega.  \label{d1c}
\end{align}
If the condition \eqref{d1c} at the boundary is replaced by
\begin{equation}
\max \{ \, -\Delta \phi(x_0), u(x_0) - g(x_0) \, \} \geq 0, x_0 \in \partial \Omega,
\end{equation}
the function $u$ is said to be a generalized viscosity supersolution of \eqref{d1}.

A function $u \in C(\tir{\Omega})$ is a viscosity solution of \eqref{d1} if it is both a viscosity subsolution and a viscosity supersolution of \eqref{d1}. Note that for this notion the boundary condition is taken in the usual sense.


\begin{theorem} \cite[Theorem 3.4]{del2018convergence} \label{class-comp}
Let $\Omega \subset \R^d$ be a bounded Lipschitz domain and $g\in C(\partial \Omega)$. If $u$ is a viscosity subsolution of \eqref{d1} and $v$ a supersolution of \eqref{d1}, then $u\leq v$ on $\tir{\Omega}$.

\end{theorem}

The authors in \cite{del2018convergence} proved the strong uniqueness property for smooth domains.

\begin{proposition} \cite{del2018convergence,Barles-Burdeau} \label{strong}
Let $\Omega \subset \R^d$ be a $C^2$ domain and $g \in C(\partial \Omega)$. Let $u$ and $v$ be respectively generalized viscosity subsolution and supersolution of \eqref{d1}. Then $u \leq v$ in $\tir{\Omega}$.
\end{proposition}

Using standard arguments we review below, the above proposition allows for $\Omega$ a $C^2$ domain to claim the convergence of the solution $u_h$ of \eqref{d11h} to the viscosity solution of \eqref{d1}. We define for a $C^2$ function $\phi$ on $\Omega_E$
\begin{align*}
S(h,x,\phi(x),\phi) &= \begin{cases} - \sum_{i=1}^d \Delta_{e_i} \phi(x), & x \in \Omega \\
  \phi(x)-G(x), & x \in O, 
\end{cases}
\end{align*}
with the operator $\Delta_e$ defined as for \eqref{Delta-e}. We write \eqref{d1} in the standard form
$S(h,x,u_h(x),u_h)=0, x \in \widetilde{\Omega}_h$ with a slight abuse of notation. We have the following analogues of \cite[(2.11)--(2.12)]{del2018convergence}.

If $u_h$ solves \eqref{d1}, we have
\begin{equation}\label{dmax}
\inf_O G \leq u_h(x) \leq \sup_O G, x \in \Omega_h.
\end{equation}
Let $u_h^1$ and $u_h^2$ solve
\begin{align*}
-\Delta_h u_h^1 & \leq 0 \text{ in } \Omega_h \text{ and } u_h^1 = G^1  \text{ on } O \\
-\Delta_h u_h^2 & \geq 0 \text{ in } \Omega_h \text{ and } u_h^2 = G^2  \text{ on } O,
\end{align*}
with $G^1 \leq G^2$. Then 
\begin{equation} \label{dcomp}
u_h^1 \leq u_h^2  \text{ on } \widetilde{\Omega}_h.
\end{equation}
Equation \eqref{dmax} says that the scheme is stable. It is a consequence of properties of the matrix of the discrete linear problem \eqref{d1} \cite[Theorem 4.77]{Hackbusch2010}. The proof is analogous to \cite[Remark 4.37]{Hackbusch2010}.

Equation \eqref{dcomp} is the discrete comparison principle. Again, it is a consequence of \cite[Theorem 4.77]{Hackbusch2010}. The proof is analogous to \cite[Theorem 4.38 b]{Hackbusch2010}.

Next, we recall the monotonicity of the scheme, i.e.  if $u_h \leq v_h$ on $\widetilde{\Omega}_h$ with $u_h(x)=v_h(x)$ we have
$S(h,x,u_h(x),u_h)\leq S(h,x,v_h(x),v_h)$.

Next, we describe the form of consistency of the scheme needed to prove convergence when the boundary condition is taken in the viscosity sense. For $x \in \tir{\Omega}$ and $\phi \in C^2(\Omega_E)$
\begin{align*}
\limsup_{h \to 0, y\to x, \zeta \to 0} S(h,y, \phi(y)+\zeta, \phi+\zeta) = \begin{cases}  -\Delta \phi(x) & \text{ if } x \in \Omega   \\
\max\{ \, -\Delta \phi(x), \phi(x)-G(x) \, \} & \text{ if } x \in \partial \Omega. \end{cases}
\end{align*}
and 
\begin{align*}
\liminf_{h \to 0, y\to x, \zeta \to 0} S(h,y, \phi(y)+\zeta, \phi+\zeta) = \begin{cases} -\Delta \phi(x) & \text{ if } x \in \Omega \\
\min\{ \, -\Delta \phi(x), \phi(x)-G(x) \, \} & \text{ if } x \in \partial \Omega. \end{cases}
\end{align*}
For $x\in \Omega$, we have the usual consistency property. For $ x \in \partial \Omega$, one can approach $x$ with points $y$ in either $\Omega$ or $O$.

\begin{theorem}
Assume that $\Omega$ is a $C^2$ domain. The solution $u_h$ of \eqref{d11h} converges uniformly on $\tir{\Omega}$ to the viscosity solution of \eqref{d1}. 
\end{theorem}

\begin{proof}
Since the scheme is stable, consistent and monotone, it follows from Proposition \ref{strong} and the framework in \cite{Barles1991} that the solution $u_h$ of \eqref{d11h} converges uniformly on $\tir{\Omega}$ to the viscosity solution of \eqref{d1}. 

\qed
\end{proof}

To handle the case $\Omega$ Lipschitz, a delicate treatment at the boundary is done in \cite{del2018convergence} using barriers on appropriate shrinking rings. Denote by $B_r(x)$ the open ball of center $x$ and radius $r$.
The following regularity condition for Lipschitz domains is the one used in the proof.

There exists $\tir{\delta} >0$ and $\mu \in (0,1)$ such that for every $\delta \in (0,\tir{\delta})$ and $y \in \partial \Omega$, there exists a ball $B_{\mu \delta}(z)$ strictly contained in $B_{\delta}(y) \setminus \Omega$. The constant $\mu$ is independent of $y \in \partial \Omega$.

We have the following analogue of \cite[Corollary 4.5]{del2018convergence}

\begin{lemma} \label{bd-lem}
Given $\eta>0$, there exist $\delta=\delta(\eta,G,\tir{\delta})$, $k_0=k_0(\eta,\mu,G)$, $h_0=h_0(\eta,\delta,\mu,k_0)$ such that
$$
|u_h(x)-G(y)| \leq \frac{\eta}{2},
$$
for all $y \in \partial \Omega, x \in B_{\delta/4^{k_0}}(y) \cap \Omega_h$ and $h\leq h_0$.
\end{lemma}


\begin{proof}

As in \cite{del2018convergence}, we prove only the one sided inequality $u_h(x) \leq G(y) + \eta/2$, the other being similar. Also, we consider only the case $d\neq 2$. The case $d=2$ is treated with similar arguments as indicated on \cite[p. 16]{del2018convergence}.

{\it Part I}: In this part we collect parts of the proof in \cite{del2018convergence}  which do not deal with dynamic programming.
Fix $\delta \in (0,\tir{\delta})$. Let $u_h$ solve \eqref{d11h}. For $y \in \partial \Omega$ define
$$
m^{h}(y) = \sup_{B_{5 \delta}(y) \cap \Gamma_{h}} G \text{ and } M^{h} = \sup_{ \Gamma_{h}} G.
$$
Define
$$
\theta =  \frac{1-\frac{1}{2} \big( \frac{\mu}{2-\mu}\big)^{\xi}-\frac{1}{2} \big( \frac{\mu}{2}\big)^{\xi}    }{1-\big( \frac{\mu}{2}\big)^{\xi}} \in (0,1) \text{ with } \xi=d-2,
$$
and for $k\geq 0$
$$
M_k^h(y) = m^{h}(y) + \theta^k( M^{h}-m^{h}(y) ).
$$
Define $\delta_k=\delta/4^{k-1}$. By the regularity assumption on $\Omega$, one can find balls $B_{\mu\delta_{k+1}}(z_k) \subset B_{\delta_{k+1}}(y) \setminus \Omega$ for all $k$. In particular $||y-z_k|| < \delta_{k+1}$.

For notational convenience, denote $m=m^{h}(y), M=M^h$ and $M_k=M_k^h(y)$. We consider the problem
\begin{align*}
\Delta U_k & = 0 \ \text{in} \ B_{\delta_k}(z_k) \setminus \tir{B}_{\mu \delta_{k+1}}(z_k) \\
U_k & = m  \ \text{on} \  \partial B_{\mu \delta_{k+1}}(z_k) \\
U_k & = M_k  \ \text{on} \  \partial B_{\delta_k}(z_k).
\end{align*}
Define
$$
a= \frac{1-\big( \frac{\mu}{2}\big)^{\xi}}{1-\big( \frac{\mu}{4}\big)^{\xi}} \text{ and } b=1-a.
$$
and (see \cite[Figure 1]{del2018convergence})
$$
\Gamma_1^h=  B_{\delta_k/2+h}(z_k) \cap \Gamma_h \text{ and } \Gamma_2^h=  ( B_{\delta_k/2+h}(z_k)
\setminus  \tir{B}_{\delta_k/2}(z_k) ) \cap \Omega.
$$
It is proven in  \cite[page 17]{del2018convergence} that
\begin{equation} \label{49}
U_k \geq a M_k + b m \ \text{ in }  \Gamma_2^h,
\end{equation}
and that for $x \in B_{(2-\mu) \delta_{k+1}}(z_k)$
\begin{equation} \label{410}
U_k(x) \leq b' m + a'M_k,
\end{equation}
where
$$
a' = \frac{1-\big( \frac{\mu}{2-\mu}\big)^{\xi}}{1-\big( \frac{\mu}{4}\big)^{\xi}} \text{ and } b' = \frac{\big( \frac{\mu}{2-\mu}\big)^{\xi}-\big( \frac{\mu}{4}\big)^{\xi}}{1-\big( \frac{\mu}{4}\big)^{\xi}}.
$$

{\it Part II}: 
We assume in this part that $u_h \leq M_k$ on $B_{\delta_k}(y) \cap \Omega_h$ for all $h<h_k$ for a fixed $h_k$ and that 
$M_k-m\geq \eta/4$. We prove that there exists $h_{k+1}=h_{k+1}(\eta,\mu,\delta,d,G) \in (0,h_{k})$ such that for all $h<h_{k+1}$, we have
\begin{equation*} 
u_h \leq M_{k+1}^h(y) \ \text{in} \ B_{\delta_{k+1}}(y) \cap \Omega.
\end{equation*}

For $h \leq \mu \delta_{k+1}/2$ the barrier $U_k$ is extended to the ring
$$
R_{k,h} = B_{\delta_k+2 h}(z_k) \setminus \tir{B}_{\mu \delta_{k+1} - 2 h}(z_k).
$$
Let $U_k^h$ be a mesh function which solves
\begin{align*}
\Delta_h U_k^h & = 0 \ \text{in} \ R_k \\
U_k^h & = U_k \ \text{in} \ R_{k,h}\setminus R_k,
\end{align*}
where $R_k=B_{\delta_k}(z_k) \setminus \tir{B}_{\mu \delta_{k+1}}(z_k)$. Note that $R_{k,h}\setminus R_k$ is the outer $h$-neighborhood of $R_k$. Since $R_k$ is smooth, by Theorem \ref{strong}, $U_k^h$ converges uniformly to $U_k$ in $R_{k,h}$ as $h \to 0$ (recall that $U_k^h=U_k$ outside $R_k$).  Therefore, given
$$
\gamma=\gamma(\mu,d,\eta) = \frac{1}{4} \frac{\big( \frac{\mu}{\mu-2}\big)^{\xi} - \big( \frac{\mu}{2}\big)^{\xi} }{1-\big( \frac{\mu}{4}\big)^{\xi}}   \frac{\eta}{4} >0,
$$
there exists $h_{k+1} = h_{k+1} (\gamma), 0<  h_{k+1} \leq \min\{ \, \mu \delta_{k+1}/2, h_k \, \}$ such that
\begin{equation} \label{loc-err}
|U_k^h-U_k| \leq \gamma,
\end{equation}
for all $h\leq h_{k+1}$ and $x \in R_{k,h}$.

On $\Gamma_1^h$, we have $u_h=G \leq m$ and so, using \eqref{loc-err}  and $\Gamma_1^h \subset R_{k,h}$ (which follows from $||y-z_k|| < \delta_{k+1}$), we get
\begin{equation*}
a u_h + b m \leq m=\inf_{R_k} U_k \leq U_k^h+ \gamma.
\end{equation*}

Since $h<\delta_{k+1}/2$ and $||y-z_k|| < \delta_{k+1}$ we have
 $B_{\delta_k/2+h}(z_k) \subset B_{\delta_k}(y)$. If $u_h \leq M_k$ on $B_{\delta_k}(y) \cap \Omega_h$, using \eqref{49} and \eqref{loc-err} we get on
$\Gamma_2^h$
$$
a u_h + b m \leq a M_k+ b m \leq U_k \leq  U_k^h+ \gamma.
$$
In summary, we have
$$
a u_h + b m \leq U_k^h+ \gamma \text{ on } \Gamma_1^h \cup \Gamma_2^h.
$$
Since the $h$-boundary of $B_{\delta_k/2}(y) \cap \Omega_h$ is contained in $\Gamma_1^h \cup \Gamma_2^h$ (see \cite[Figure 1]{del2018convergence}), by the discrete comparison principle \eqref{dcomp}, we obtain
$$
a u_h + b m \leq U_k^h+ \gamma  \text{ on } B_{\delta_k/2}(y) \cap \Omega_h.
$$
Using again \eqref{loc-err} we have under the assumption $u_h \leq M_k$ on $B_{\delta_k}(y) \cap \Omega_h$
\begin{equation} \label{48}
a u_h + b m \leq U_k+ 2 \gamma  \text{ on } B_{\delta_k/2}(y) \cap \Omega_h.
\end{equation}

Next, since $B_{\delta_{k+1}}(y) \subset B_{\delta_k/2}(z_k)$ we get by \eqref{48}
\begin{equation} \label{411}
a u_h + b m \leq U_k+ 2 \gamma  \text{ on } B_{\delta_{k+1}}(y) \cap \Omega_h.
\end{equation}

As $B_{\delta_{k+1}}(y) \subset  B_{(2-\mu) \delta_{k+1}}(z_k)$, we have by \eqref{410} $a u_h + b m \leq b' m + a'M_k+ 2 \gamma$ on $B_{\delta_{k+1}}(y) \cap \Omega_h$. Since by assumption $M_k-m\geq \eta/4$, we have 
$$\gamma \leq \frac{1}{4} \frac{\big( \frac{\mu}{\mu-2}\big)^{\xi} - \big( \frac{\mu}{2}\big)^{\xi} }{1-\big( \frac{\mu}{4}\big)^{\xi}} (M_k-m).
$$
Using in addition
$b'-b+a'=a$, we obtain for $h<h_{k+1}$ in $B_{\delta_{k+1}}(y) \cap \Omega_h$
\begin{multline} \label{412}
u_h \leq \frac{b'-b}{a} m + \frac{a'}{a} M_k +\frac{2 \gamma}{a} \leq m+\frac{a'}{a} (M_k-m) + \frac{b'(M_k-m)}{2 a} \\
=m+\theta (M_k-m) = m+\theta^{k+1} (M-m).
\end{multline}

{\it Part III}: As $G$ is uniformly continuous on the compact set $\Gamma_1$, there exists $0<\delta < \tir{\delta}$ such that for all $y \in \partial \Omega$,
\begin{equation} \label{mh}
m^h(y)-G(y)< \frac{\eta}{4}.
\end{equation}
By the stability result \eqref{dmax}
$$
u_h \leq \max_{x \in \partial \Omega_h} G \leq \sup_{\Gamma_h} G = M^h,
$$
for all $h$. Thus we can take $h_0=1$.

If $M^h_0(y)-m^h(y)=M^h-m^h(y) < \eta/4$, then by \eqref{mh}, we obtain $u_h \leq M^h <m^h(y)+\eta/4 < G(y)+\eta/2$. Otherwise, there exists $0<h_1<h_0$ such that for all $h<h_1$, $u_h\leq M_{1}^h(y)$ in $B_{\delta_{1}}(y) \cap \Omega$.
If $M^h_1-m^h(y) < \eta/4$ we proceed as before to obtain the desired inequality. 

As $0<\theta<1$, for some integer $s$, we have $ \theta^{s}( M^{h}-m^{h}(y) ) <\eta/4$ which with \eqref{mh} gives the desired inequality if $u_h \leq M^h_{s}$ on $B_{\delta_{s}}(y) \cap \Omega$. If for some $k<s$ we have $M^h_k-m^h(y) < \eta/4$, we take $k_0=k$. Otherwise $k_0=s$.
\qed
\end{proof}
We can now state

\begin{theorem}
Assume that $\Omega$ is a Lipschitz domain. The solution $u_h$ of \eqref{d11h} converges uniformly on compact subsets of $\Omega$ to the viscosity solution of \eqref{d1}. 
\end{theorem}

\begin{proof}
The proof is as on \cite[p. 18]{del2018convergence}. It follows from Lemma \ref{bd-lem} that the half relaxed limits satisfy the boundary condition in the classical sense. One can then use Theorem \ref{class-comp} and the framework in \cite{Barles1991}.
\qed
\end{proof}

\begin{acknowledgements}
The author would like to thank the referees for their comments which lead to a better presentation with simpler proofs. In
particular the statement of Lemmas \ref{equivalence} and \ref{cvg-cvx} are due to one of the referees as well as Theorem \ref{equi-measures} and its
proof. 
The proof of Lemma \ref{unif-disc-conv} given was suggested by one of the referees. The author was partially supported by NSF grants DMS-1319640 and  DMS-1720276. 
The author would like to thank the Isaac Newton Institute for Mathematical Sciences, Cambridge, for support and hospitality during the programme ''Geometry, compatibility and structure preservation in computational differential equations'' where part of this work was undertaken. Part of this work was supported by EPSRC grant no EP/K032208/1.
\end{acknowledgements}




\end{document}